\documentclass[11pt]{amsart}
\usepackage{epsfig,graphics}
\usepackage{amscd}
\usepackage{amsmath}
\usepackage{bm}
\usepackage{amssymb}
\usepackage{graphicx}
\usepackage{psfrag}

\newtheorem{thm}{Theorem}

\newtheorem{lemma}{Lemma}
\newtheorem{cor}{Corollary}
\newtheorem{con}{Conjecture}
\newtheorem{prob}{Problem}
\newtheorem{prop}{Proposition}
\newtheorem{defi}{Definition}
\newtheorem{rem}{Remark}

\newcommand{\CA}{{\rm CA}}
\newcommand{\MA}{{\rm MA}}
\newcommand{\A}{{\mathcal{A}}}

\newtheorem{innercustomthm}{Theorem}
\newenvironment{customthm}[1]
  {\renewcommand\theinnercustomthm{#1}\innercustomthm}
  {\endinnercustomthm}

\newtheorem{innercustomcon}{Conjecture}
\newenvironment{customcon}[1]
  {\renewcommand\theinnercustomcon{#1}\innercustomcon}
  {\endinnercustomcon}

\begin{document}

\title{On Sloane's persistence problem}

\author{Edson de Faria}
\address{Instituto de Matem\'atica e Estat\'istica, Universidade de S\~ao Paulo}
\curraddr{Rua do Mat\~ao 1010, 05508-090, S\~ao Paulo SP, Brasil}
\email{edson@ime.usp.br}

\author{Charles Tresser}
\address{IBM, P.O.~Box 218}
\curraddr{Yorktown Heights, NY 10598, USA}
\email{charlestresser@yahoo.com}

\thanks{This work has been supported by ``Projeto Tem\'atico Din\^amica em Baixas Dimens\~oes'' FAPESP Grant 2011/16265-2, and 
by FAPESP Grant 2012/19995-0}



\begin{abstract} We investigate the so-called persistence problem of Sloane, exploiting connections with the dynamics 
of circle maps and the ergodic theory of $\mathbb{Z}^d$ actions.  
We also formulate a conjecture concerning the asymptotic distribution of digits in long products of finitely many 
primes whose truth would, in particular, solve the persistence problem. The heuristics that we propose to 
complement our numerical studies can be thought in terms of a simple model in statistical mechanics.
\end{abstract}

\maketitle

\section{Introduction}\label{intro}

In \cite{S}, Sloane proposed the following curious problem. Take a non-negative integer, write down its decimal representation, and multiply 
its digits together, getting a new non-negative integer. Repeat the process until a single-digit number is obtained. The problem can thus be stated: 
Is the number of steps taken in this process uniformly bounded?

\subsection{General formulation} Let us start with a general formulation of Sloane's problem, 
while at the same time introducing some of the notation that we will use. 
Given a natural number $n$, and an {\it integer base\/} $q>1$, 
consider the base-$q$ expansion of the number $n$, say
\begin{equation}\label{ninbaseb}
 n\;=\;\left[d_1d_2\cdots d_k\right]_q \;=\; \sum_{j=1}^k d_j q^{k-j}\ ,
\end{equation}
where each digit $d_j\in \{0,1,\dots, q-1\}$ (and $d_1\neq 0$ when $n\geq 1$). Let $S_q(n)$ denote the product of all such digits, {\it i.e.\/},
\[
 S_q(n)\;=\;\prod_{j=1}^k d_j\ .
\]
Thus $n\mapsto S_q(n)$ defines a map $S_q:\mathbb{Z}^+\to \mathbb{Z}^+$, which we call the {\it Sloane map\/
in base $q$}. 
Clearly, such map can be iterated: 
write $S_q(n)$ in base $q$, multiply 
its digits to obtain $S_q(S_q(n))$, and so on. 
In particular, given any $n\in \mathbb{Z}^+$ we can consider its {\it orbit\/} under the Sloane map, namely
\[
 n\, ,\,S_q(n)\, , \, S_q^2(n)\, , \, \ldots \, , \, S_q^m(n)\, , \, \ldots
\]
The following proposition ensures that this sequence always stabilizes after a finite number of steps. 

\begin{prop}\label{trivialstart}
We have $S_q(n)<n$ for all $n\geq q$ (i.e., as long as the base $q$ expansion of $n$ has at least two digits).
\end{prop}  

\begin{proof}
Write $n$ in base $q$ as in \eqref{ninbaseb}, and note that $k>1$. Since $d_j\leq q-1$ for all $j$, it follows that
\[
 S_q(n)\;=\;d_1\cdot \prod_{j=2}^k d_j\;\leq\; d_1\cdot (q-1)^{k-1} \;<\;d_1\,q^{k-1}\;\leq\; n\ .
\]
\end{proof}

From Proposition \ref{trivialstart} we deduce that $n$ is a fixed point of $S_q$ if and only if $k=1$.  
It also follows from Proposition \ref{trivialstart} that every orbit of 
$S_q$ is finite and converges to some $d<q$ that is a fixed point.
In other words, there exists a minimum number $\nu_q(n)$ such that $S_q^i(n)=S_q^{\nu_q(n)}(n)$ for all
$i \geq \nu_q(n)$. Hence $\nu_q(n)$  {\it is the smallest number $m$ such that $S_q^{m}(n)$ has a single digit\/}.  
Sloane asked in \cite{S} whether such minimum number of steps until a fixed point is uniformly bounded.  
The number  $\nu_q(n)$  is known as the {\it persistence\/}{\footnote{What we call persistence in this paper is sometimes referred to as 
{\it multiplicative persistence\/} elsewhere, to distinguish it from the similarly defined concept of {\it additive persistence\/}, 
introduced by Hinden \cite{H}. Since we will only consider multiplicative persistence, we will have no use for the adjective.}}
 of $n$ in base $q$.
Numerical evidence that $\nu_q(n)$ is bounded has been collected for some values of $q$.  
Furthermore, the answer to Sloane's question is trivially positive for $q\,=\,2$ since for any 
$n\geq 0$ one has  $S_2(n)\in\{0,1\}$, and $\{0,1\}$ is the fixed-point set of $S_2$. 
The problem -- known as the {\it persistence problem\/} -- can be stated as follows.

\begin{prob}\label{prob1}
For a given $q>2$, is there a positive number $B(q)$ such that $\nu_q(n)\leq B(q)$ for all $n\/$ ? 
 \end{prob}

A related set of issues goes as follows (considering now $B(q)=\sup_n\nu_q(n)$ as an element of $\mathbb{Z}^+\cup\infty$). 

\begin{prob}\label{prob2}
What is the behavior of $B(q)$ seen as a function of $q$? More precisely, one can ask: 
\begin{enumerate}
\item[(a)] Is the answer to Problem \ref{prob1} positive for all, or all but finitely many, or most, or infinitely many, or perhaps  
only finitely many values of $q$?
\item[(b)]  What is the asymptotic behavior of $B(q)$  as $q\to\infty$? 
\end{enumerate}
 \end{prob}

Here are some known facts about the persistence problem in various bases:
\begin{enumerate}
 \item In base $q\,=\,2$, the situation  is rather trivial: every positive integer has persistence $1$ in base $2$.
 \item In base $q=3$, no number with persistence
greater than $3$ has ever been found. 
 \item In base $q=10$, the number $n=68889$ has persistence $7$, because under the Sloane map $S_{10}$ we have
\[
 68889 \mapsto 27648 \mapsto 2688 \mapsto 768 \mapsto 336 \mapsto 54 \mapsto 20 \mapsto 0
\]
In fact, this is the smallest number with persistence equal to $7$.
 \item Still in base $q=10$, the number $n=277777788888899$ has persistence $11$. It is the smallest number with persistence equal to $11$. 
 \item It is conjectured that $\nu_{10}(n)\leq 11$ for all $n$. This has been checked for all $n$ up to $10^{233}$ .
\end{enumerate}

\subsection{Goals and endeavors}
In this paper, we have two main goals. The first goal is   
to examine the persistence problem in the light of some Dynamical Systems considerations. 
We will show that Sloane's question (Problem \ref{prob1}) has an affirmative answer in a certain {\it probabilistic\/} sense. 
Roughly speaking, we will show that {\it for any base $q$, the set of natural numbers $n$ with persistence $\geq 3$, {\it i.e.\/} such that $S_q^2(n)\neq 0$, 
is an extremely rarified subset of $\mathbb{Z}^+$\/}. The probabilistic sense in question 
will be made progressively clear in \S \ref{sec:erg1} and \S \ref{sec:erg2}. 

We will see in particular that Problem \ref{prob1} has a positive answer for $q=3$ if a precise orbit that we will fully describe has a ``generic'' 
behavior under the $\mathbb{Z}$-action determined by a well-defined piecewise affine degree one circle map. Similarly, Problem \ref{prob1} has a positive 
answer for $q=4$ if two precise orbits that we will fully describe have a ``generic'' behavior under the $\mathbb{Z}$-action determined by another 
well-defined piecewise affine degree one circle map. The affine circle maps that we will encounter here are defined by $q$ and a number $p<q$, 
a digit in base $q$. 

For bases $q>4$, the relevant dynamical systems for the Sloane map are no longer $\mathbb{Z}$-actions, but rather  
${\mathbb{Z}}^k$-actions with $k>1$. More precisely, they are given by certain free abelian groups of piecewise affine degree-one circle maps. 
We will exploit some simple ergodic properties of such 
free-abelian actions in order to derive our main probabilistic result on the Sloane map, namely Theorem \ref{rankkdensity}. 

Our second goal is 
to formulate a very general conjecture, namely Conjecture \ref{conj:convvtoequi}, concerning the asymptotic distribution of digits in the base $q$ 
expansion of long products 
whose factors are chosen from a given finite set of primes. This conjecture is 
conveniently 
formulated in terms of certain objects 
that we call {\it multiplication automata\/}, in part because their time evolution produces patterns that resemble those produced by the evolution of 
(one-dimensional) cellular automata. 
The most general form of the conjecture, as stated in \S \ref{multaut},  is too broad to allow 
thorough numerical tests. Thus, short of a proof, we felt the need to provide  
an heuristic explanation lending support to this conjecture. Such heuristics is given in the 
language of {\it inhomeogeneous Markov chains\/}, as a form of convergence to equilibrium for the 
evolution of our automata. The multiplication automata and the associated convergence to fair 
distribution may be of interest to the physics of growth processes, and perhaps to other aspects 
of statistical mechanics.

\subsection{The Erd{\"o}s-Sloane map}
In order to avoid possible misunderstandings, we warn the reader that another map has been studied, inspired by $S_q$: let us call it the 
Erd{\"o}s-Sloane map, and denote it by 
$S^*_q$. For each natural number $n$, $S^*_q(n)$ is the product of the non-zero digits of $n$ in base $q$. 
According to Guy \cite{G}, this map was introduced by Erd{\"o}s.  
Since our approach in this paper is based on the conjectured 
generic property that all digits should become equally probable for (products of) high powers of digits, we have nothing new to 
say about $S^*_q$ (although some results we have about $S_3$ can be interpreted in terms of $S^*_3$).

\section{Conjectures, remarks, and first simple results}

\subsection{A trivial remark and some conjectures}
Let us start by making a trivial but important remark. Choose some base $q>1$. 
Looking back at the expansion given in \eqref{ninbaseb}, it is clear that we have the following {\it trivial dichotomy}:

\begin{enumerate}
 \item {\it Either\/} one of the digits $d_j$ is equal to zero, in which case $S_q(n)=0$;
 \item {\it Or else\/} $S_q(n)$ is equal to a product of digits in $1,2,\dots, q-1$.
\end{enumerate}

This trivial dichotomy, 
together with first observations of relevant but scarce numerical data, 
suggest 
us an obvious strategy.  
In order to answer Sloane's question in the affirmative for base $q$, it suffices to establish the following.

\begin{con}\label{basebpersist}
 There exists a positive integer $k_0(q)$ such that, for all $k\geq k_0(q)$, the base $q$ 
 expansion of any product of $k$ digits greater than $1$ for base $q$ has at least one digit equal to zero.
\end{con}

As soon as a product of digits becomes divisible by $q$, a zero appears as its right-most digit when written in base 
$q$, and therefore the next iteration of the Sloane map yields $0$ as the result. This trivial remark allows us to concentrate 
only on those products that are not divisible by the base $q$. 
The effect of that will be dramatic when $q=4$. The following two conjectures are implied by Conjecture \ref{basebpersist}. 

In the case of base $q=3$, the only non-zero values assumed by the Sloane map are powers of $2$. In other words, we only need to 
investigate the orbit of $1$ under the doubling map $x\mapsto 2\cdot x$. 

\begin{customcon}{1a}\label{base3persist}
There exists a positive integer $k_{3,2}$ such that, for all $k\geq k_{3,2}$, the base-$3$ expansion of $2^k$ has at least one digit equal to zero.
\end{customcon}

Such a statement is reminiscent of a conjecture by Erd{\"o}s to the effect that there is always a 2 among the digits in 
base 3 of $2^k$ for all $k$ sufficiently large. This question has been recently adressed by Lagarias in \cite{L}.    

In the case of base $q=4$, in principle all products of the form $2^m\cdot 3^n$ would need to be examined.  
But since the basis 4 divides $2^m\cdot 3^n$ as soon as $m\geq 2$, we only have to consider: (a) powers of 3 or 
(b) products of the form $2\cdot 3^m$. In other words, we only need to consider the orbits of 1 and 2 under the 
tripling map $x\mapsto 3\cdot x$.

\begin{customcon}{1b}\label{base4persist}
There exists a positive integer $k_{4,3}$ such that, for all $k\geq k_{4,3}$, the base-$4$ expansions of $3^k$ 
and $2\cdot 3^k$ each have at least one 
digit equal to zero.
\end{customcon}

There is considerable computational evidence in favor of these conjectures.
Proving these conjectures is certainly sufficient to solve the persistence problem for the corresponding bases.

\subsection{A weak estimate on $\nu_3(n)$}
Let us take the time to establish a weak estimate on $\nu_3(n)$, the minimum stability time of $n\in \mathbb{Z}^+$ (in base $3$) 
introduced in \S \ref{intro}. This should be compared to a similar estimate proved by Erd\"os concerning the Erd\"os-Sloane map.  
First, a very simple lemma.

\begin{lemma}\label{trivlem1}
 If $n>3$, then at least one of the digits in the base-$3$ expansion of $2^n$ is not equal to $2$.
\end{lemma}
 \begin{proof}
Suppose the base-$3$ expansion of $2^n$ has exactly $k$ digits, all equal to $2$. Then $2^n=3^k-1$, and we have a solution to Catalan's equation
$3^x-2^y=1$ with $y=n>3$, which is impossible. 
{\footnote{It is not necessary to use the highly non-trivial result about the full Catalan's conjecture (concerning the Diophantine equation 
$a^x-b^y=1$ --see \cite{M}). The special case needed here (with bases $2$ and $3$, which are {\it prime\/}) can be proved by elementary means -- 
see \cite[p.~85]{LeV}.}}
 \end{proof}

\begin{rem}\label{trivrem1}
 It follows from the proof that $S_3(2^n)\leq 2^{k-1}$. 
\end{rem}

Let us now present our weak estimate. Recall that if $N\in \mathbb{Z}^+$, then the number of digits in the base-$q$ 
expansion of $N$ is equal to $1+ \lfloor \log_q{N}\rfloor$. 

\begin{prop}
 For all $n\geq 3$, we have $\nu_3(n)\leq 2(1+\log_3\log_3{n})$.
\end{prop}

\begin{proof}
 Let $n>3$ be given, and assume $S_3(n)\neq 0$, otherwise there is nothing to prove. Thus, suppose $n_1,n_2,\ldots,n_m$ are such that
\begin{enumerate}
 \item[(i)] $n_j>3$ for $j=1,2,\ldots, m$;
 \item[(ii)] $S_3(n)=2^{n_1}$;
 \item[(iii)] $S_3(2^{n_j})=2^{n_{j+1}}$ for $j=1,2,\ldots, m-1$;
 \item[(iv)] $m$ is {\it maximal\/} with the above properties.
\end{enumerate}
From these properties we see that the only possible values for $S_3(2^{n_m})$ are $0,1,2,2^2$ or $2^3$. This implies that $\nu_3(n)\leq 3+ m$. 
Hence it suffices to find a suitable bound for the number $m$. 

Let $k_j$ denote the number of digits in the base-$3$ expansion of $n_j$, for $j=1,2,\ldots,m$. Since $n_j>3$, we deduce from Lemma \ref{trivlem1}
and Remark \ref{trivrem1} that $S_3(2^{n_j})=2^{n_{j+1}}\leq 2^{k_j-1}$. But we also know that $k_j=1+ \lfloor n_j\log_3{2}\rfloor$. Therefore we have
\[
 n_{j+1}\;\leq\; \alpha n_j\ , \ \ \ \textrm{for}\ \ j=1,2,\ldots, m_1\ ,
\]
where $\alpha=\log_3{2}$. From this it follows that
\[
 3\;<\;n_m\;\leq\; \alpha^{m-1} n_1\ .
\]
But we also have $n_1\leq \log_3{n}$. Hence, extracting the base-$3$ logarithm of all terms in the above inequality, we see that
\[
 m\;<\;1 + \frac{(\log_3\log_3{n}) -1}{\log_3{(\alpha^{-1})}} \;<\; -1+2\log_3\log_3{n} \ .
\]
This shows that $\nu_3(n)<2(1+\log_3\log_3{n})$, as claimed.
\end{proof}

\section{Ergodic $\mathbb{Z}$-actions and persistence in bases $q=3,4$}\label{sec:erg1}

In this section, more precisely in \S \ref{subsec:ergoz} below, we introduce a simple dynamical system -- generated by a single 
piecewise affine degree-one circle map -- which will turn out to be very useful in the study of the Sloane map for bases $q=3,4$. 
But first let us present some simple arithmetical facts.

\subsection{Periodicity of tails}\label{sec:tails}

Looking at the base-$3$ expansions of the successive powers of two (see Figure \ref{fig:PowersOf2Base3}), it is clear that for every $k\geq 1$ 
the rightmost $k$ digits of these powers exhibit a periodic behavior. Our goal here is to calculate the minimum period. 
This periodicity is more general: the last $k$ digits 
of the base-$q$ expansion of $p^n$ form a periodic sequence, for all $1<p<q$. This fact is a simple consequence of modular arithmetic, 
and we call it {\it tail periodicity\/}.

First, we need the following simple lemma. Recall that, if $q>1$ is an integer, then the set of all residues $r\in \{0,1,2,\ldots, q-1\}$ modulo $q$ that 
are relatively prime 
with $q$ form a {\it multiplicative group\/} (under multiplication modulo $q$), called the {\it group of units modulo $q$\/}, denoted $U_q$. 
This group $U_q$ has 
order $\phi(q)$, where $\phi$ is Euler's {\it totient function\/}. An integer $p$ is said to be a {\it primitive root modulo $q$\/} if the 
residue class of $p$ modulo $q$ has order $\phi(q)$ in $U_q$ (in particular, if such a primitive root exists, then the group $U_q$ must be cyclic).

\begin{lemma}\label{primroot1}
 Let $p$ and $q$ be integers such that $1<p<q$ and 
 $p$ is a primitive root modulo $q$. Then the residues $p^n (\!\!\mod q)$, $n=0,1,2,\ldots$, form a periodic sequence 
with minimum period equal to $\phi(q)$.
\end{lemma}

\begin{proof}
 By hypothesis, $p\in U_q$ is a generator of the cyclic group $U_q$. Therefore, by Euler's theorem, $p^{\phi(q)}\equiv 1 (\!\!\mod q)$, and the elements
$1,p,p^2,\ldots,\break p^{\phi(q)-1}$ exhaust the elements of $U_q$. Since $p^{n+\phi(q)}\equiv p^n (\!\!\mod q)$ for all $n$, the sequence $p^n$ is indeed periodic, 
with minimum period $\phi(q)$. 
\end{proof}

As a consequence, we have the following. Given integers $k\geq 1$ and $n\geq 0$, let $r_n(k)$ denote the residue of $2^n$ modulo $3^k$.

\begin{lemma}\label{primroot2}
 For each $k\geq 1$, the sequence $r_n(k)$, $n=0,1,\ldots$, is periodic with minimum period $2\cdot 3^{k-1}$.
\end{lemma}

\begin{proof}
 It is well-known that $2$ is a primitive root modulo $3^k$ for all $k\geq 1$ (see \cite[p.~81]{LeV}). Thus, the 
desired result follows at once from 
lemma \ref{primroot1} (with $p=2$ and $q=3^k$) and the fact that $\phi(3^k)=3^{k}-3^{k-1}=2\cdot 3^{k-1}$. 
\end{proof}

Note that for each $k$ there are exactly $2\cdot 3^{k-1}$ strings of length $k$ in the symbols $0,1,2$ 
which do not end in a zero. 
Call these strings {\it allowable\/}. Combining this counting, the above lemma and the fact that $r_n(k)$ also never ends in a zero 
(otherwise $2^n$ would be divisible by $3$), we see that in any full (minimum) period of $r_n(k)$ each allowable string appears exactly once. 
Moreover, since there are exactly $2^k$ allowable strings of length $k$ that show only the digits $1$ and $2$, the proportion of allowable strings 
in which at least one zero appears (relative to the total number of allowable sequences of length $k$) tends to $1$ as $k\to\infty$. 
From this we deduce that the set $A=\{n\in \mathbb{Z}^+:\;S_3(2^n)=0\}$ has {\it asymptotic density $1$\/}, in the sense of \S \ref{sec:density} below. 
This fact can be generalized to other bases (but the counting argument used here becomes a bit more involved). Rather than delve into such 
arithmetic methods, we shall use a completely different approach to prove a more general density result that holds true for all bases.

\subsection{Almost-periodicity of heads}

If, instead of looking at the final $k$ digits of the base-$3$ expansions of powers of two, we look at the first 
$k$ digits (for $k\geq 1$ fixed), the behavior of such {\it heads\/} is no longer periodic. 
It is rather {\it almost periodic\/}, as we will show (see Remark \ref{remalmost}).

\begin{figure}[t]
\begin{center}~
\hbox to \hsize{ \psfrag{0}[][][1]{$0$} \psfrag{q}[][][1]{$\frac{1}{q}$}
\psfrag{p}[][][1]{$\frac{1}{p}$} \psfrag{pq}[][][1]{$\frac{p}{q}$}
\psfrag{1}[][][1]{$1$}
\includegraphics[width=3.2in]{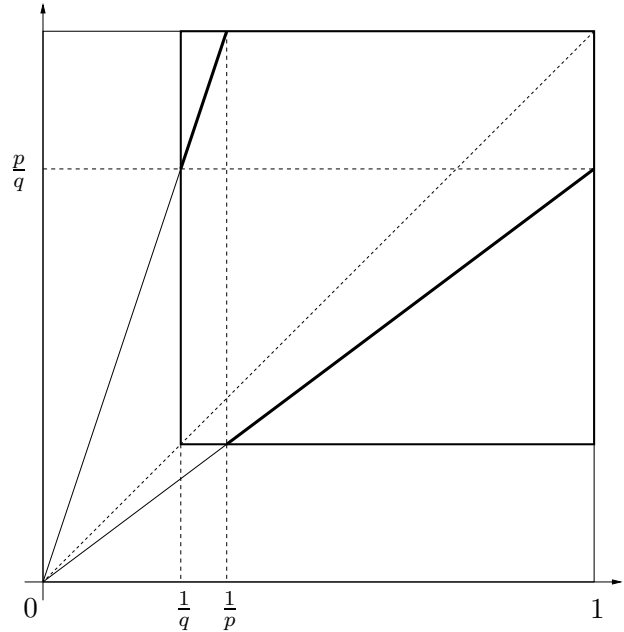}
   }
\end{center}
  \caption[slope]{\label{slope} A piecewise linear circle map built out of two linear maps, one expanding with slope $p$, 
the other contracting with slope $\frac{p}{q}$.}
\end{figure}

\subsection{An ergodic $\mathbb{Z}$-action}\label{subsec:ergoz}
In fact, let us present a more general result. We consider two integers $1<p<q$ and consider the sequence $p^n$, $n=0,1,2,\ldots$, written 
in base $q$, say
\begin{equation}\label{baseq}
 p^n\;=\; \left[d_{1,n}d_{2,n}\dots d_{k_n,n}\right]_q\;=\;\sum_{i=1}^{k_n} d_{i,n}q^{k_n-i}\ ,
\end{equation}
where $0\leq d_{i,n}\leq q-1$ and $d_{1,n}\neq 0$, and where the number of digits $k_n$ is given by $k_n=\lceil n\log_q{p}\rceil$. 
If we add a ``decimal point'' in front of the above expansion, we get the number
\begin{equation}\label{alternate}
 x_n\;=\;\frac{p^n}{q^{k_n}}\;=\;\left[0.d_{1,n}d_{2,n}\cdots d_{k_n,n}\right]_q\;\in\;\Delta_q \ ,
\end{equation}
where $\Delta_q=[\frac{1}{q},1]\subset \mathbb{R}$. Note that $x_0=(0.1)_q= 1/q$. Now, the crucial observation is that the sequence $(x_n)$ 
just defined is precisely the orbit of $x_0=1/q$ under the piecewise-affine map $T_{p,q}: \Delta_q\to \Delta_q$ given by
\begin{equation}\label{paffineT}
 T_{p,q}(x)\;=\; \left\{\begin{matrix}
                   px & , & \textrm{if}\ \;\dfrac{1}{q}\leq x < \dfrac{1}{p}\ ,\\
                   {} & {} & {}\\
                   \dfrac{p}{q}x & , & \textrm{if}\ \;\dfrac{1}{p}\leq x\leq 1\ .
                  \end{matrix} \right.
\end{equation}
This map is built out of two linear maps of the real line: one expanding (multiplication by $p>1$), the other contracting (multiplication 
by $p/q<1$); see Figure \ref{slope}. Through the identification of the endpoints of $\Delta_q$, the map $T_{p,q}$ becomes a (piecewise affine) homeomorphism 
of the circle. The reader can easily check that $x_n = T_{p,q}^n(x_0)$, for all $n\geq 0$.
{\footnote{There are some analogies between Sloane's question, at least for base 3, and the Collatz 
conjecture as discussed in Terence Tao's blog entry \cite{T}. In both cases, one represents the problem 
by a question about a map  on a 
compact abelian group (here the circle, while one uses the dyadics for the Collatz conjecture). 
Also, in both cases one 
can formulate a main aspect of the problem in terms of powers of 2 in base 3: here by expressing that 
$2^n$ has at least one zero in its base 3 expansion for all $n>15$, and as explained by Tao in said blog 
for the Collatz case.  
}}

Now we have the following result.

\begin{thm}\label{proptopconj}
The map $T_{p,q}: \Delta_q\to \Delta_q$ is topologically conjugate to the rotation by $\alpha=\log_q{p}$. 
The conjugacy is Lipschitz, in fact 
differentiable except at
a single point. Moreover, $T_{p,q}$ has an absolutely continuous invariant measure given by
\[
 d\mu(x)\;=\; \frac{dx}{x\log{q}} \ .
\] 
\end{thm}

\begin{proof}
 Let us write $T=T_{p,q}$ in this proof. First note that 
\begin{align*}
 T^n(x_0)\;&=\;\frac{p^n}{q^{\lceil n\alpha\rceil}}\;=\;\frac{1}{q^{1-\{n\alpha\}}}\left(\frac{p}{q^\alpha}\right)^n \\
           &=\; q^{\{n\alpha\}-1} \ ,
\end{align*}
where we have used that $pq^{-\alpha}=1$. Hence, defining $h: [0,1]\to \Delta_q$ by 
\begin{equation}\label{hconj}
 h(t)\;=\;q^{t-1}\;=\; \frac{1}{q}\exp\{ t\log{q}\} \ ,
\end{equation}
we see that $h(\{n\alpha\})=T^n(x_0)$. In other words, $h$ maps the orbit of $0$ under the rotation $R_\alpha: t\mapsto t+\alpha\, (\!\!\mod 1)$ 
onto the orbit of $x_0$ under $T$. This suggests that $h$ is a conjugacy between $R_\alpha$ and $T$. 

To see this, let $t\in [0,1]$, and note that there are two possibilities.
\begin{enumerate}
 \item[(i)] We have $0\leq t< 1-\alpha$: in this case, we have $R_\alpha(t)=t+\alpha$, and therefore
\begin{align*}
 h\circ R_\alpha(t)\;&=\; q^{t+\alpha -1}\;=\;q^{\alpha}h(t) \\
                     &=\; q^{\log_q{p}}h(t)\;=\; p h(t)\;=\; T\circ h(t)\ .
\end{align*}
 \item[(ii)] We have $1-\alpha\leq t\leq 1$: in this case, $R_\alpha(t)=t+\alpha -1$, and therefore
\begin{align*}
 h\circ R_\alpha(t)\;&=\; q^{t+\alpha -2}\;=\;q^{\alpha -1}h(t) \\
                     &=\; q^{\log_q{p}-1}h(t)\;=\; \frac{p}{q} h(t)\;=\; T\circ h(t)\ .
\end{align*}
\end{enumerate}
Thus, it follows that $h\circ R_\alpha=T\circ h$. The map $h$ is a homeomorphism, as is clear from \eqref{hconj}. In fact $h$ is an 
analytic diffeomorphism in the open interval $(0,1)$. Upon identification of the endpoints $0$ and $1$, $h$ becomes a circle homeomorphism which is 
differentiable at all points except one, where a break in the derivative occurs. 

The absolutely continuous invariant measure $\mu$ for $T$ in $\Delta_q$ can be obtained very simply as the push-forward of Lebesgue measure $\lambda$ 
in $[0,1]$ via $h$. To wit, if $E\subseteq \Delta_q$ is Borel measurable, we have
\[
 \mu(E)\;=\;\lambda(h^{-1}(E))\;=\; \int_{h^{-1}(E)} dt\;=\; \int_{E} (h^{-1})'(x)\,dx\ .
\]
Since from \eqref{hconj} we have
\[
 h^{-1}(x)\;=\; \frac{\log{qx}}{\log{q}}\ ,
\]
it follows that
\[
 d\mu(x)\;=\;(h^{-1})'(x)\,dx\;=\; \frac{1}{x\log{q}}\,dx\ ,
\]
as was to be proved.
\end{proof}

\begin{rem}\label{remalmost}
Note that we do {\it not\/} assume that $\alpha=\log_q{p}$ is irrational in the above theorem.
When $\alpha$ is irrational, the orbit of any point under $T_{p,q}$ is almost periodic. Hence we have the promised almost periodicity of tails. 
For a more general result concerning piecewise affine homeomorphisms of the circle, see \cite{Li}.
\end{rem}

\begin{rem}\label{largerp}
The above theorem also holds true when $p>q$, provided we define $T_{p,q}$ appropriately in such cases. 
This can be done as follows. Let $k\geq 1$ be such that $q^k<p<q^{k+1}$, and define 
\begin{equation}\label{paffineT2}
 T_{p,q}(x)\;=\; \left\{\begin{matrix}
                   \dfrac{p}{q^k}x & , & \textrm{if}\ \;\dfrac{1}{q}\leq x < \dfrac{q^k}{p}\ ,\\
                   {} & {} & {}\\
                   \dfrac{p}{q^{k+1}}x & , & \textrm{if}\ \;\dfrac{q^k}{p}\leq x\leq 1\ .
                  \end{matrix} \right.
\end{equation}
Then, just as before, $T_{p,q}: \Delta_q\to \Delta_q$ is a piecewise affine degree-one 
circle map with rotation number $\alpha=\{\log_q{p}\}$, 
and this circle map preserves the same measure $\mu$ given in Theorem \ref{proptopconj} ({\it cf.\/} \S \ref{sec:further}). 
\end{rem}

\begin{rem}
In \cite{Bo}, Boshernitzan gave an example of a piecewise affine homeomorphism of the unit circle 
$\mathbb{R}/\mathbb{Z}$ (seen as the interval $[0,1]$ with the endpoints identified) 
which preserves the rationals and has only dense orbits. We remark that our $T_{p,q}$'s yield 
a {\it plethora of such examples\/}. Indeed, defining $L:[0,1]\to \Delta_q$ by 
$L(t)=\frac{1}{q}+ \left(1-\frac{1}{q}\right)t$, we see that each conjugated map $L^{-1}\circ T_{p,q}\circ L$ is an
example of the type devised by Boshernitzan, provided its rotation number $\alpha=\{\log_q{p}\}$ is irrational. 
\end{rem}

Let us now present two corollaries of Theorem \ref{proptopconj}.
\subsection{A density result}\label{sec:density}
In order to state and prove the first corollary, we recall that the {\it lower density\/} 
of a subset $A\subseteq \mathbb{Z}^+$, is given by
\[
 D^{-}(A)\;=\;\lim\inf_{n\to \infty} \frac{1}{n}\#(A\cap [1,n])\ .
\]
One defines the {\it upper density\/} $D^{+}(A)$ of $A$ in similar fashion, replacing the $\lim\inf$ by $\lim\sup$. 
If the upper and lower densities are equal, then we say that $A$ has {\it asymptotic density\/} $D(A)=D^{+}(A)=D^{-}(A)$. 
Recall also that a continuous self-map $T$ of a compact metric space $X$ is {\it uniquely ergodic\/}, {\it i.e.\/} 
has a unique invariant Borel probability measure $\mu$, if and only if its Birkhoff time averages 
$\frac{1}{n}\sum_{i=0}^n f\circ T^i$ converge uniformly to the space average $\int_X f\,d\mu$,   
for every $f\in C^0(X)$ (see for instance \cite[p. 160]{W}).  

\newpage

\begin{cor}\label{rankonedensity}
 If $1<p<q$ are such that $\log_q{p}$ is irrational, then the set $A=\{n\in \mathbb{Z}^+\,:\; S_q(p^n)=0\,\}$ has asymptotic density equal to $1$.
\end{cor} 
 
\begin{proof}
Let $T=T_{p,q}$ and $\mu$ be as in Theorem \ref{proptopconj}, and let $x_n=T^n(x_0)$ be our special orbit as before.  
Since $T$ is Lipschitz-conjugate 
to an irrational rotation, it is uniquely ergodic. 
For each $j\geq 1$, let $B_j\subseteq \Delta_q$ be the set of all points $x$ whose base-$q$ expansion 
has at least one digit $0$ among its first $j$ digits. Then $B_j$ is a finite union of 
intervals, and its Lebesgue measure is easily seen to be
\[
 \lambda(B_j)\;=\; |\Delta_q|\left[1-\left(1-\frac{1}{q}\right)^{j-1}\right]
\]
Thus, $\lambda(B_j)\nearrow |\Delta_q|$, and therefore $\mu(B_j)\nearrow 1$, as $j\to\infty$. 
Note that if $x_n\in B_j$ for some $n$ with $k_n\geq j$, then $S_q(p^n)=0$, {\it i.e.\/} $n\in A$. 
Since $B_j$ is a {\it finite\/} union of intervals and $\mu$ is a regular Borel measure, we can find  
a continuous function $f_j:\Delta_q\to \mathbb{R}$  such that $0\leq f_j\leq \chi_{B_j}$ everywhere and
\begin{equation}\label{uniq1} 
 \int_{\Delta_q}f_j\,d\mu \;\geq\; \mu(B_j)-\frac{1}{2j}\ .
\end{equation}
In particular, for all $N\geq 1$ we have
\begin{equation}\label{uniq2}
 \frac{1}{N}\sum_{n=0}^{N-1} \chi_{B_j}\circ T^n(x_0) \;\geq\; \frac{1}{N}\sum_{n=0}^{N-1} f_j\circ T^n(x_0)\ .
\end{equation}
Using \eqref{uniq1} and the fact that $T$ is uniquely ergodic, we deduce that
\begin{equation}\label{uniq3}
 \frac{1}{N}\sum_{n=0}^{N-1} f_j\circ T^n(x_0)\;\geq\; \int_{\Delta_q} f_j\,d\mu -\frac{1}{2j} \;\geq\; \mu(B_j)-\frac{1}{j}\ ,
\end{equation}
provided $N$ is sufficiently large. Hence, combining \eqref{uniq2} and \eqref{uniq3}, we get
\begin{equation}
 \frac{1}{N}\sum_{n=0}^{N-1} \chi_{B_j}(x_n)\;\geq\;\mu(B_j)-\frac{1}{j} \ ,
\end{equation}
for all sufficiently large $N$. 
Writing $A_j=\{n\in \mathbb{Z}^+:\,x_n\in B_j\}$, we have just proved that 
\[
 D^{-}(A_j)\;\geq\; \mu(B_j)-\frac{1}{j}
\]
Since all but finitely many elements of $A_j$ belong to $A$, it follows that
$D^{-}(A)\geq \mu(B_j)-\frac{1}{j}$ as well, for every $j$. Letting $j\to\infty$, we deduce that
$D(A)=D^{-}(A)=1$, as desired. 
\end{proof}

\subsection{A formula for the digits of $p^n$ in base $q$}
Another consequence of Theorem \ref{proptopconj} is the following explicit formula for the $j$-th 
digit of $p^n$ written in base $q$. 
\begin{cor}
 If $d_{j,n}$ denotes, as in \eqref{baseq}, the $j$-th digit from left to right in the base $q$ expansion of $p^n$, then
\begin{equation}\label{digitclosedform}
 d_{j,n}\;=\; \left\lfloor q\left\{q^{j+\{n\alpha\}-2}\right\} \right\rfloor \ ,\ \textrm{for all}\ j=1,2,\ldots,k_n\ ,
\end{equation}
where, as before, $\alpha=\log_q{p}$ and $k_n$ is the number of digits in that expansion.
\end{cor}

\begin{proof}
 Again, we let $T=T_{p,q}$. Since $d_{j,n}$ is also the $j$-th digit of $x_n=T^n(x_0)$ after the decimal point, we can certainly write
\[
 d_{j,n}\;=\; \left\lfloor q\left\{q^{j-1}x_n\right\} \right\rfloor \ .
\]
But by Theorem \ref{proptopconj}, we have
\[
 x_n\;=\;T^n(x_0)\;=\; h\circ R_{\alpha}^n(0)\;=\;h(\{n\alpha\})\;=\;q^{\{n\alpha\} -1} \ .
\]
This immediately implies formula \eqref{digitclosedform}. 
\end{proof}

\begin{rem}
Although it may seem a bit surprising that we have such an explicit formula for the digits $d_{j,n}$, the formula 
is in practice rather useless for large values of $n$: roughly speaking, evaluating $d_{j,n}$ depends on knowing at
least the first $n$ digits of $\alpha=\log_q{p}$ after the decimal point. 
\end{rem}

At least for bases $3$, $4$, $5$, and $10$, we have strong computational evidence suggesting the validity of the following conjecture, which 
(for base $3$) is certainly stronger than conjecture \ref{base3persist}.

\begin{con}\label{equaldigits}
If $q$ is not a power of $p$, then for each $d=0,1,\ldots, q-1$, we have 
\[
 \lim_{n\to\infty} \frac{1}{k_n} \,\# \left\{ 1\leq j \leq k_n:\ d_{j,n}=d\,\right\}     
\;=\; \frac{1}{q} \ .
\]
(As before, $k_n=\lceil n\log_q{p}\rceil$ is the number of digits of $p^n$ in base $q$.)
\end{con}

For computational evidence and heuristic arguments supporting this conjecture, see \S \ref{sec:compevid}. 
If $q=p^k$ for some $k\geq 1$, then for every $n\geq 1$ the base $q$ expansion 
of $p^{nk}$ consists of the digit $1$ followed by $n$ zeros. This shows that the hypothesis that $q$ is not a power of $p$ is indeed necessary. 

\section{Ergodic $\mathbb{Z}^d$-actions and persistence in base $q>4$}\label{sec:erg2}

We wish to generalize the results presented in \S \ref{sec:erg1} to the cases when the base $q$ is greater than $4$.
The relevant dynamical system here is no longer the (semi-)group given by a single piecewise affine degree-one circle map, but rather 
the (semi-)group generated by several such maps. 

\subsection{Abelian actions for the Sloane map}\label{sec:abelact}
Let $2=p_1<p_2<\cdots<p_m\leq q-1$ be the list of all primes smaller than $q$. If $n\in \mathbb{Z}^+$ is such that 
$S_q(n)\neq 0$, then we can certainly write
\begin{equation}\label{sqprod}
 S_q(n)\;=\; \prod_{i=1}^m p_i^{n_i} \ ,
\end{equation}
where $(n_1,n_2,\cdots,n_m)\in \mathbb{Z}_{+}^m = \left(\mathbb{Z}^+\right)^m\subset \mathbb{Z}^m$. 
We can treat the base-$q$ 
expansion of the right-hand side of \eqref{sqprod} pretty much in the same way as we treated the digits of 
$p^n$ in base $q$ in \S \ref{subsec:ergoz}. Keeping the notation introduced in \S \ref{subsec:ergoz}, 
let $\Delta_q=[q^{-1},1]\equiv \mathbb{S}^1$, and consider the piecewise affine homeomorphisms 
$T_{p_i}=T_{p_i,q}: \Delta_q\to \Delta_q$ ($i=1,2,\ldots,m$) defined taking $p=p_i$ in \eqref{paffineT}. 
We denote by $PL^+(\Delta)$ the group of all piecewise affine homeomorphisms of the interval 
$\Delta\subseteq \mathbb{R}$.

\begin{lemma}
 The group $G_q\subset PL^+(\Delta_q)$ generated by $\{T_{p_i}:\,i=1,2,\ldots,m\}$ is abelian. 
\end{lemma}

\begin{proof}
 Let $h: [0,1]\to \Delta_q$ be the homeomorphism constructed in the proof of Theorem \ref{proptopconj}. 
Then for each $i$ we have $T_{p_i}=h\circ R_{\alpha_i}\circ h^{-1}$, where $\alpha_i=\log_q{p_i}$ and 
$R_{\alpha_i}:\,x\mapsto x+\alpha_i \, ({\mod 1})$ is the corresponding rotation. Since any two circle rotations commute, we have 
$R_{\alpha_i}\circ R_{\alpha_j}=R_{\alpha_j}\circ R_{\alpha_i}$, and therefore $T_{p_i}\circ T_{p_j}=T_{p_j}\circ T_{p_i}$ as 
well, so $G_q$ is abelian. 
\end{proof}

This lemma tells us in particular that we have a well-defined surjective homomorphism $\mathbb{Z}^m\to G_q$ given by
\[
 \mathbb{Z}^m \ni {\bm{n}}=(n_1,n_2,\ldots,n_m) \; \mapsto \; T^{\bm{n}}\;=\; T_{p_1}^{n_1}\circ T_{p_2}^{n_2}\circ\cdots\circ T_{p_m}^{n_m} \; \in G_q \ .
\]
However, this homomorphism is not necessarily one-to-one (but see below). In any case, we see that $\mathbb{Z}^m$ acts 
on the circle in a special 
way as a group of piecewise-affine homeomorphisms of the circle. 
What is the relevance of this action to the study of the Sloane map in base $q$? In order to answer this question, 
we proceed as in \S \ref{subsec:ergoz}. Let 
$x_0=[0.1]_q= \frac{1}{q}\in \Delta_q$. Then for each $\bm{n}\in \mathbb{Z}_{+}^m$ we have
\begin{equation}\label{fracorbit}
 T^{\bm{n}}(x_0)\;=\; \frac{p_1^{n_1}p_2^{n_2}\cdots p_m^{n_m}}{q^{k(\bm{n})}} \ ,
\end{equation}
where $k(\bm{n})$ is the number of digits of the base-$q$ expansion of the numerator, given by
\[
 k(\bm{n})\;=\; \left\lceil \sum_{i=1}^m n_i\log_q{p_i}\right\rceil\;=\; \left\lceil \sum_{i=1}^m n_i\alpha_i\right\rceil \ .
\]
Thus, we see that the entire range of values assumed by the Sloane map is contained in a single orbit of the action of 
the semi-group $\mathbb{Z}_{+}^m \subset \mathbb{Z}^m$.

\subsection{Detour: ergodic free-abelian actions} We shall need the following facts about {\it free-abelian\/} actions. 
Let $G$ be a free abelian group of rank $k$, with a fixed set of generators $\{e_1,e_2,\ldots,e_k\}$. Each $g\in G$ has a unique 
representation $g=\sum n_i e_i$ with $n_i\in \mathbb{Z}$ for all $i=1,2,\ldots,k$. We write $\|g\|=\max_{1\leq i\leq k} |n_i|$, 
the {\it norm\/} of $g$. We denote by $G^+\subset G$ the semigroup consisting of all elements $g$ for which $n_i\geq 0$ for all $i$. 
For each $N\in \mathbb{Z}^+$, let $\Lambda_N(G)=\{g\in G: \|g\|\leq N\}$, and let $\Lambda_N(G^+)=\Lambda_N(G)\cap G^+$. 

Now suppose that the group $G$ (or the semigroup $G^+$) acts on a probability measure space $(X,\mu)$ as a group (or semigroup) of 
measure-preserving transformations (m.p.t.'s). In other words, if $T^g:X\to X$ denotes the m.p.t. associated to $g\in G$, then $T^g_*\mu=\mu$ 
for every $g$ (where the star denotes push-forward of measures). We say that the $G$-action (or $G^+$-action) on $(X,\mu)$ is {\it ergodic\/} if 
the only measurable subsets $E\subseteq X$ that are invariant under $G$ (or $G^+$) -- {\it i.e.\/} such that $(T^g)^{-1}E\subseteq E$ for all $g$ -- 
are either null-sets or full-measure sets. Just as for rank-one measure-preserving actions, there is a multi-dimensional version 
of Birkhoff's ergodic theorem, both in the case of free-abelian groups and free-abelian semigroups. 
We state the version for semigroups, which is the relevant one here.

\begin{customthm}{A}\label{multiplebirkhoff}
 If the $G^+$-action on $(X,\mu)$ is ergodic, then for every $f\in L^1(X,\mu)$ and for $\mu$-almost every $x\in X$ we have
\begin{equation}\label{multibirk1}
 \lim_{N\to\infty} \frac{1}{\# \Lambda_N(G^+)} \sum _{g\in \Lambda_N(G^+)} f\circ T^g(x) \;=\; \int_X f\,d\mu\ .
\end{equation}
\end{customthm}

A proof of this theorem can be found in \cite[ch.~2]{K}. 

One can also define {\it unique ergodicity\/} by analogy with the rank-one case. A $G$-action (or $G^+$-action) on a compact metric space 
$X$ through continuous maps is {\it uniquely ergodic\/} if there exists a unique Borel probability measure on $X$ which is $G$-invariant 
(or $G^+$-invariant). As for rank-one actions, we have the following fact (which, once again, we state only for semigroup actions).

\begin{customthm}{B}\label{uniqueergod}
 If a $G^+$-action by continuous maps on a compact metric space $X$ is uniquely ergodic 
then for every $f\in C(X)$ 
\[
 \frac{1}{\#\Lambda_N(G^+)} \sum _{g\in \Lambda_N(G^+)} f\circ T^g\ \ \textrm{converges uniformly to}\ \ \ \int_X f\,d\mu
\]
as $N\to\infty$, where $\mu$ is the unique $G^+$-invariant Borel probability measure.
\end{customthm}

The proof of the analogous statement for rank-one actions as given in, say, \cite[pp.160-161]{W} applies {\it mutatis mutandis\/} 
to the present case. The converse of Theorem \ref{uniqueergod} is also true, but will not be needed here.

\subsection{First density result} 
Now we go back to our investigation of the Sloane map. 
Given a positive integer $N$, let $\Lambda_N$ denote the ``cube'' $\Lambda_N=\{{\bm{n}}\in \mathbb{Z}^m: |n_i|\leq N,\, i=1,2,\ldots,m\}$, 
and let $\Lambda_N^+=\Lambda_N \cap \mathbb{Z}_{+}^m$. The {\it lower asymptotic density\/} of a subset 
$A\subseteq \mathbb{Z}_{+}^m$ is defined to be
\[
 D^{-}(A)\;=\; \lim\inf_{N\to \infty} \frac{1}{N^m} \#(A\cap \Lambda_N^+) \ .
\]
The {\it upper asymptotic density\/} of $A$, denoted $D^{+}(A)$, is similarly defined (replacing $\lim\inf$ by $\lim\sup$). We always have
$0\leq D^{-}(A)\leq D^{+}(A)\leq 1$. When $D^{-}(A) = D^{+}(A)$, this common value is called the {\it asymptotic density\/} of $A$ 
and it is denoted by $D(A)$. 

\begin{prop}\label{trivdensity}
 If the base $q$ is not a prime number, then the set
\[
 A\;=\; \left\{ {\bm{n}}=(n_1,n_2,\ldots,n_m)\in \mathbb{Z}_{+}^m: \, S_q(p_1^{n_1}p_2^{n_2}\cdots p_m^{n_m})=0\,\right\}
\]
has asymptotic density equal to $1$. 
\end{prop}

\begin{proof}
Since $q$ is not a prime, we can certainly write $q= p_1^{a_1}p_2^{a_2}\cdots p_m^{a_m}$ where each $a_i\geq 0$. Now, if 
$\bm{n}\in \mathbb{Z}_{+}^m$ is such that $n_i\geq a_i$ for all $i$, then $q$ divides $p_1^{n_1}p_2^{n_2}\cdots p_m^{n_m}$, and therefore
$S_q(p_1^{n_1}p_2^{n_2}\cdots p_m^{n_m})=0$. This shows that $A\supseteq \{\bm{n}\in \mathbb{Z}_{+}^m:\, n_i\geq a_i \ \textrm{for all}\;i\}$. 
Hence, for every $N$ sufficiently large we have
\[
 \#\left(A\cap \Lambda_N^+\right)\;\geq\; \prod_{i=1}^{m} (N-a_i+1) \ ,
\]
and therefore
\[
 D^{-}(A)\;\geq\; \lim_{N\to\infty}\,\frac{1}{N^m} \prod_{i=1}^{m} (N-a_i+1)\;=\;1
\]
This proves that $D(A)=D^{-}(A)=1$. 
\end{proof}

The set $A$ in Proposition \ref{trivdensity} has asymptotic density $1$ for trivial reasons: most $m$-tuples 
$(n_1,n_2,\ldots,n_m)$  are such 
that $n_i\geq a_i$ for all $i$. Note that we did not exploit the abelian action introduced in \S \ref{sec:abelact}. We will prove 
in \S \ref{sec:secdensity} below a 
more refined version of Proposition \ref{trivdensity} using the ergodic properties of such abelian actions.

\subsection{Second density result}\label{sec:secdensity}

Our second density result makes use of the group $G_q$ and its action on the circle.  
In fact, it will make use of certain subgroups of $G_q$, which turn out to be {\it free abelian\/}. 

We will need the following lemma. Given $1\leq i_1<i_2<\cdots<i_k\leq m$, let us denote
the subgroup of $G_q$ generated by $\{T_{p_{i_1}}, T_{p_{i_2}},\ldots, T_{p_{i_k}}\}$ by 
$G(i_1,i_2,\ldots,i_k)$. 
Recall that $\alpha_i=\log_q{p_i}$ is the rotation number of $T_{p_i}$ 

\begin{lemma}\label{freeuniqergod}
 If the numbers $1,\alpha_{i_1},\alpha_{i_2},\ldots, \alpha_{i_k}$ are rationally independent 
{\footnote{In other words, linearly independent over the field of rational numbers}}, then 
$G(i_1,i_2,\ldots,i_k)$ is a free abelian group of rank $k$. Moreover, its action on the circle $\Delta_q\equiv \mathbb{S}^1$ 
is uniquely ergodic. 
\end{lemma}

\begin{proof}
Suppose $(n_1,n_2,\ldots,n_k)\in \mathbb{Z}^k$ is such that 
\[ 
T_{p_{i_1}}^{n_1} \circ T_{p_{i_2}}^{n_2} \circ \cdots \circ T_{p_{i_k}}^{n_k}\;=\; Id \ .
\]
From this and the fact that each $T_{p_i}$ is conjugate to $R_{\alpha_i}$ by the same conjugating map, we have
\[
 R_{\alpha_{i_1}}^{n_1}\circ R_{\alpha_{i_2}}^{n_2}\circ \cdots \circ R_{\alpha_{i_k}}^{n_k} \;=\; Id
\]
But then $n_1\alpha_{i_1}+n_2\alpha_{i_2}+\cdots + n_k\alpha_{i_k}\equiv 0 \;(\!\!{\mod 1})$. In other words, 
there exists $N\in \mathbb{Z}$ such that $n_1\alpha_{i_1}+n_2\alpha_{i_2}+\cdots + n_k\alpha_{i_k} =N$. 
The hypothesis of rational independence implies that $n_1=n_2=\cdots =n_k=N=0$. Hence there are no 
non-trivial relations in $G(i_1,i_2,\ldots,i_k)$. This shows the group is free abelian as stated. 
Moreover, at least one of the $\alpha_{i_j}$'s must be irrational. Say $\alpha_{i_1}$ is irrational; 
then $T_{p_{i_1}}$, being conjugate to an irrational rotation, is uniquely ergodic. Therefore, 
{\it a fortiori\/}, the action of $G(i_1,i_2,\ldots,i_k)$ on the circle is uniquely ergodic.
\end{proof}

We are now in a position to state and prove our second density result. This result refines Proposition \ref{trivdensity}, 
and in particular covers the case when the base $q$ is prime. 

\begin{thm}\label{rankkdensity}
 Let $1\leq i_1<i_2<\cdots<i_k\leq m$ be chosen so that the numbers $1,\alpha_{i_1},\alpha_{i_2},\ldots, \alpha_{i_k}$ are rationally independent. 
Then for every divisor $d$ of $q$, the set
\[
 A\;=\;\left\{ (n_1,n_2,\ldots,n_k)\in \mathbb{Z}_+^k : \ S_q(p_{i_1}^{n_1} p_{i_2}^{n_2}\cdots p_{i_k}^{n_k}\cdot d)=0\,\right\}
\]
has asymptotic density equal to $1$. 
\end{thm}

\begin{proof}
Things have been set up so that the same argument used in the proof of Corollary \ref{rankonedensity} can be applied, {\it mutatis mutandis\/}. 
By Lemma \ref{freeuniqergod}, the action of the group $G=G(i_1,i_2,\ldots,i_k)\cong \mathbb{Z}^k$ on the circle $\Delta_q\equiv \mathbb{S}^1$ 
is uniquely ergodic; the unique invariant measure is the measure $\mu$ constructed in Theorem \ref{proptopconj}. 
Let us fix the divisor $d$ of the base $q$. 
We are interested in a particular orbit of the semigroup $G^+\cong \mathbb{Z}_+^k$, namely that of the point $w_0=d/q\in \Delta_q$. 
For each $\bm{n}\in \mathbb{Z}_+^k$, let us write $w_{\bm{n}}=T^{\bm{n}}(w_0)$. Also, denote by $\ell(\bm{n})$ the number of digits 
in the base-$q$ expansion of the number $p_{i_1}^{n_1} p_{i_2}^{n_2}\cdots p_{i_k}^{n_k}\cdot d$. Then, just as in \eqref{fracorbit}, we have
\begin{equation}\label{fracorbit2}
 w_{\bm{n}}\;=\; \frac{p_{i_1}^{n_1} p_{i_2}^{n_2}\cdots p_{i_k}^{n_k}\cdot d}{q^{\ell(\bm{n})}} \ .
\end{equation}
As in the proof of Corollary \ref{rankonedensity}, for each $j\geq 1$, let $B_j\subseteq \Delta_q$ be the set of all points $x$ whose base-$q$ expansion 
has at least one digit $0$ among its first $j$ digits. As we saw there, $\mu(B_j)\nearrow 1$ as $j\to\infty$. 
Let $f_j: \Delta_q\to \mathbb{R}$ be as in that proof also; thus, each $f_j$ is continuous 
and $0\leq f_j \leq \chi_{B_j}$, and 
\begin{equation}\label{newuniq1} 
 \int_{\Delta_q}f_j\,d\mu \;\geq\; \mu(B_j)-\frac{1}{2j}\ .
\end{equation}
The point now -- as is clear from \eqref{fracorbit2} -- is that
\[
 S_q(p_{i_1}^{n_1} p_{i_2}^{n_2}\cdots p_{i_k}^{n_k}\cdot d)=0 \;\iff \; T^{\bm{n}}(w_0)=w_{\bm{n}}\in B_j \ \textrm{for some}\ j\leq \ell({\bm{n}})\ .
\]
In other words, $\bm{n}\in A$ if and only if $w_{\bm{n}}\in B_j$ for some $j\leq k_{\bm{n}}$. By analogy with what we did in the proof of 
Corollary \ref{rankonedensity}, for each $j\geq 1$ we define $A_j=\{\bm{n}\in \mathbb{Z}_+^k:\; w_{\bm{n}}\in B_j\}$. Note that all but finitely many elements 
of $A_j$ belong to $A$. Hence, in order to prove that $D(A)=1$, it suffices to show that $D^{-}(A_j)\nearrow 1$ as $j\to \infty$. 
Since the action of $G^+$ on the circle is uniquely ergodic, combining Theorem \ref{uniqueergod} with \eqref{newuniq1} we deduce that, for each 
fixed $j$ and all sufficiently large $N$,  
\[
 \frac{1}{\#\Lambda_N^+} \sum _{\bm{n}\in \Lambda_N^+} f_j\circ T^{\bm{n}}(w_0)\;\geq\; \int_{\Delta_q}f_j\,d\mu - \frac{1}{2j}\;\geq\; \mu(B_j)-\frac{1}{j}\ .
\]
Since $f_j\leq \chi_{B_j}$, it follows that
\[
 \frac{1}{\#\Lambda_N^+} \sum _{\bm{n}\in \Lambda_N^+} \chi_{B_j}(w_{\bm{n}})\;\geq\; \mu(B_j)-\frac{1}{j}\ ,
\]
for all sufficiently large $N$. Letting $N\to \infty$, this shows that
\[
 D^{-}(A_j)\;\geq \; \mu(B_j) - \frac{1}{j}
\]
for all $j$, and therefore $D^{-}(A_j)\nearrow 1$ as $j\to \infty$, as required. 
\end{proof}

\subsection{Further generalizations}\label{sec:further}
Since in this paper we are primarily interested in the Sloane map, in all of the above we have focused on products of 
prime numbers smaller than the base $q$. 
However, most of what we have done goes through when some or all of such primes are greater than $q$. 
Recall from Remark \ref{largerp} that if $p>q$ we can still define the circle maps $T_{p,q}$, and these still commute with each other because 
they all share a common absolutely continuous invariant measure. 
Thus, suppose that $F$ is some non-empty finite set of primes, and that not all the prime divisors of $q$ are in $F$. 
Then the set $\{1\}\cup \{\log_q{p}:\,p\in F\}$ is rationally independent. Therefore the maps $T_{p,q}$ for $p\in F$ generate a 
free-abelian group $G_F$ and the action of $G_F$ on the circle $\Delta_q$ is uniquely ergodic 
(as in Lemma \ref{freeuniqergod}, and the proof is the same). Moreover, the density result given in Theorem \ref{rankkdensity} also holds true 
for the list of primes in $F$. In particular, the vast majority of products of the form $\prod_{p\in F} p^{n_p}$ (with $n_p\in \mathbb{Z}^+$) 
have at least one digit zero in their base-$q$ expansions. In \S \ref{sec:compevid} we will make a much stronger assertion, in the form of a 
conjecture, to the effect that the digits in the base-$q$ expansions of such products become asymptotically equidistributed as 
$\max\{n_p:\, p\in F\}\to \infty$.

\section{Computational evidence and heuristics}\label{sec:compevid}

\subsection{Asymptotic distribution of digits in long products of primes}

To ease many statements to be made in rest of the paper, we will call {\it Sloane's conjecture\/} the positive answer to the question 
about multiplicative persistence raised by Sloane in 1973. Our goal is to place Sloane's conjecture as a consequence of a
much more general conjecture concerning long products of primes (chosen from a given finite set).

As in \S \ref{sec:further}, let us be given a base $q>1$ and a finite set of primes $F$ with the property that not all prime 
divisors of $q$ belong to $F$. 
We call such an $F$ an {\it allowable set of primes for $q$\/}.
Suppose we play the following game: starting with 
any given positive integer $a$, we randomly select a sequence $\pi_1,\pi_2,\ldots, \pi_n,\ldots$ of primes in $F$ and we use them to generate the 
sequence  of products $N_n=a\cdot\pi_1\cdot\pi_2\cdots\pi_n$  with $n\in \mathbb{Z}^+$. Then it turns out that, regardless of the initial 
seed $a$ and of the sequence of primes selected, the 
frequency of each digit $d\in\{0,1,\ldots,q-1\}$ in $N_n$ always seems to approach $\frac{1}{q}$ as $n\to\infty$.  More seems to be true, 
and we formulate the conjecture that was revealed by our numerical computations in the following elementary way. 

\begin{customcon}{3}\label{conj:convvtoequi0} 
(Elementary formulation).
Given an integer $q>1$, an allowable set of primes $F$ for $q$, and a positive integer 
$a$, consider any of the possible sequences of products defined by setting $N_0=a$ and then, for each $n\geq0$, 
$N_{n+1}=\pi_{n+1}\cdot N_n$ where each $\pi_i$ is a element of $F$. Then the digits $\{0,1,2,\dots,q-1\}$ are asymptotically 
equidistributed (their numbers tend  to be in equal proportions) when $n\to\infty$ in the base-$q$ representations of these 
successive $N_n$'s. Furthermore, the same holds true for blocks of consecutive digits of any length, \emph{i.e.,} the 
asymptotic proportion of each block of digits of length ${\ell}>0$ is given by $\frac{1}{q^{\ell}}$, the reciprocal of the number of 
distinct blocs of length ${\ell}$ in base $q$. 
\end{customcon}

We remark that when the restriction of allowability is removed (but $F$ remains finite) it seems that the non-zero digits remain well 
distributed while the zeros may be much more abundant, which is
more than we need 
in order to get the Sloane conjecture as a
corollary.  

\subsection{Multiplication automata}\label{multaut} Let us give a more precise and more general statement of the above conjecture, expressing it in the language 
of {\it automata\/}. The patterns arising by the multiplication game described above strongly suggest analogy with 
the dynamics of some simple \emph{cellular automata} (abbreviated $\CA$)
 \emph{i.e.,} automorphisms of a shift space that commute with the shift \cite{He}. 
Successsive multiplications by any $p>1$ (or by any sequence of such $p$'s) do not yield a $\CA$, because of the {\it carryover effect\/}, 
but give rise to a close enough object, 
which we define as follows.   
 
\begin{defi}\label{quasicellular}
 Let $q>1$ and let $F$ be an allowable set of primes for $q$. A multiplication automaton (or $\MA$), with alphabet $\A_q = \{0,1,\ldots,q-1\}$ 
consists of a finite sequence of primes $\pi_n\in F$ (with $n\geq 1$) called the multipliers, and two maps, the configuration map 
$x:\mathbb{Z}_+^2\to \A_q$ and the carryover map $c:\mathbb{Z}_+^2\to \A_q$ satisfying the following rules
for all $n\geq 1$ and all $i\geq 1$:
\begin{enumerate}
 \item[(i)] $x_{i,n}=\pi_n\cdot x_{i,n-1} + c_{i-1,n}\mod q$;
 \item[(ii)] $\displaystyle{c_{i,n} = \frac{1}{q}(\pi_n\cdot x_{i,n-1} + c_{i-1,n} - x_{i,n})}$.
\end{enumerate}
\end{defi}
Here, we assume that the initial row $(x_{i,0})_{i\in\mathbb{Z}^+}$ of the configuration map is given, as well as the column 
$(c_{0,n})_{n\in \mathbb{Z}^+}$ of initial carryovers. We also assume that $x_{i,0}=0$ for all but finitely many values of $i$, 
and we call the number $a=\sum_{i\geq 0} x_{i,0}q^i$ the {\it seed\/} of the $\MA$.  Note that each row of an automaton has only finitely 
many non-zero elements; in other words, for each $n\in \mathbb{Z}^+$ there exists a smallest $k_n\geq 0$ such that $x_{i,n}=0$ for all $i\geq k_n$. 
If the initial seed is non-zero, then $k_n\to\infty$ as $n\to\infty$. 

Despite appearances from the above recursive formulas, it turns out that the values of the configuration map of a $\MA$ are determined 
{\it purely locally\/},  as the following proposition shows.  
This further reinforces the similarity of $\MA$'s with $\CA$'s. {\footnote{Warning: The word ``locally'' is used here with a 
different meaning from the one used when studying, e.g.,cellular automata; in that other context, the carryover is indeed a very 
non-local effect.}}

\begin{prop}\label{prop:locality}
 In every $\MA$, the configuration value $x_{i,n}$ depends only on the three values $x_{i,n-1},\, x_{i-1,n-1},\, x_{i-1,n}$ 
and on the multiplier $\pi_n$. 
\end{prop}
\begin{proof}
 We refer to formulas (i) and (ii) in Definition \ref{quasicellular}. From (ii) with $i-1$ replacing $i$ we see that
\begin{equation}\label{localone}
 c_{i-1,n}\;=\;\frac{1}{q}(\pi_n x_{i-1,n-1} + c_{i-2,n} - x_{i-1,n}) \ .
\end{equation}
But formula (i) with $i-1$ replacing $i$ gives us
\begin{equation}\label{localtwo}
 c_{i-2,n}\;=\; (x_{i-1,n} - \pi_n x_{i-1,n-1})\!\!\!\!\! \mod q
\end{equation}
Combining \eqref{localone} and \eqref{localtwo} with (i) we deduce that 
\[
 x_{i,n}= \pi_n x_{i,n-1} + \frac{1}{q}\left[\pi_n x_{i-1,n-1}  + ((x_{i-1,n} - \pi_n x_{i-1,n-1})\!\!\!\!\!\!  \mod q) - x_{i-1,n} \right] \, ,
\]
which is the desired result. 
\end{proof}

\begin{rem}
The way the $\MA$ operates purely locally as told by Proposition \ref{prop:locality} 
is illustrated in part (a) of Figure \ref{fig:qauto} where in each group of positions, the one framed in black is the value of $x_{i,n}$.
It is worth pointing out that $\MA$'s can be interpreted as examples of \emph{error diffusion}  where the input is the previous line, 
the modified input is the product of the input by $\pi_n$ added to the carry over, and the new error is the next carry over 
(see \cite{Tr} and references therein). Incidentally, part (b) of the same figure gives the allowed configuration of another automaton, 
namely the one that describes the evolution along columns (a simple model of discrete epitaxy).  
The configurations in part (c) of that figure are the forbidden ones resulting in a solid black square when nothing can be computed.    
\end{rem}
 
\subsubsection*{Example 1}
The simplest example of a multiplication automaton relevant to Sloane's conjecture is obtained by taking $q=3$ and $F=\{2\}$, 
choosing $a=1$ as the seed, and setting $c_{0,n}=0$ for all $n$ as the initial sequence of carryovers. Note that in this case 
$\pi_n=2$ for all $n$. The configuration map for this automaton yields the list of all powers of $2$ written in base $3$.
This list corresponds to the orbit of $\frac{1}{3}=[0.1]_3$ under the map $T_{2,3}$ introduced in \S \ref{subsec:ergoz}.
In the notation introduced above, we have $2^n=\sum_{i=0}^{k_n-1} x_{i,n}q^i$ for all $n$. 
Figure 2 exhibits the first 45 rows of this $\MA$, {\it i.e.\/} the first 45 powers of 2 (out of the about 8000 that we have computed).  
Some patterns seen there  are reminiscent of those appearing in the simplest $\CA$'s.

\subsubsection*{Example 2}
Another pair of examples is obtained by taking $q=4$ and $F=\{3\}$, setting $c_{0,n}=0$ for all $n$ as before, and choosing either 
$a=1$ or $a=2$ as the seed. The resulting pair of $\MA$ yields the two sequences $(3^n)_{n\geq 0}$ and $(2\cdot 3^n)_{n\geq 0}$ 
written in base $4$, which are the relevant ones for Sloane's conjecture in that base. These sequences correspond to the orbits 
of $\frac{1}{4}=[0.1]_4$ and $\frac{1}{2}=[0.2]_4$ under successive iterations of the map $T_{3,4}$ introduced in \S \ref{subsec:ergoz}. 
The 23 first rows of these automata are presented side by side in Figure 3. 

A $\MA$ as given in Definition \ref{quasicellular} can be used as a statistical model for the 
discrete-time evolution of a mixture of $q$ species (labeled balls, say, or distinct molecules) with positions at time $n$  labeled by 
the integers $\{0,1,\dots , k_n-1\}$.  One is then interested in the asymptotic behavior of the mixture as characterized by the proportions 
of species populations as $n\to\infty$.  This analogy with statistical mechanics and the observed patterns in the two examples above 
(see Figures 2 and 3), as well as in several other examples we have investigated, suggest the following conjecture.
{\footnote{This conjecture is indeed more general than its elementary counterpart given earlier: the latter corresponds to the cases 
when the initial sequence of carryovers in the $\MA$ is identically zero.}}

\begin{customcon}{3}\label{conj:convvtoequi} 
(General formulation.)
The configuration map of every multiplication automaton with non-zero seed converges to an equilibrium in the following sense.
For each $\ell>0$ and each $\ell$-block $b_1b_2\cdots b_\ell\in \mathcal{A}_q^{\ell}$, we have
\[
 \lim_{n\to\infty} \frac{1}{k_n}\# \left\{ 0\leq i\leq k_n-\ell:\; x_{i,n}=b_1,\, x_{i+1,n}=b_2,\, \ldots ,\,x_{i+\ell-1,n}=b_\ell\right\}
\,=\, \frac{1}{q^{\ell}}\ .
\]
\end{customcon} 
 
In other words, the populations of the $q$ species in the mixture become asymptotically perfectly balanced as time evolves. 
Figure 4 shows clearly that the proportion of zeros in $2^n$ ({\it i.e.\/}, in row $n$ of the automaton of Example 1) goes to 
one third of the digits when $n$ becomes 
large enough. Of course there are fluctuations in the proportions but the sizes of these fluctuations go to zero, 
as one would expect in a model of statistical physics. Similar data have been obtained with bases $4$, $5$, and $10$ and a variety of 
allowable sets $F$ of prime multipliers. In many cases, we also checked that the expected asymptotic statistics of {\it blocks\/} behave 
according to the Conjecture.

Besides potential applications to random number generators, Conjecture \ref{conj:convvtoequi}, if proved true, would yield several interesting corollaries, 
such as Sloane's conjecture for 
all bases 
(as we have already pointed out), and other conjectures by Erd{\"o}s, Furstenberg, and Lagarias that are reported in \cite{L}.  

\phantom{Here is some space at the top}

\begin{figure}[ht]
\hbox to \hsize {
{\hspace{-8.0em} \includegraphics[width=10cm]
{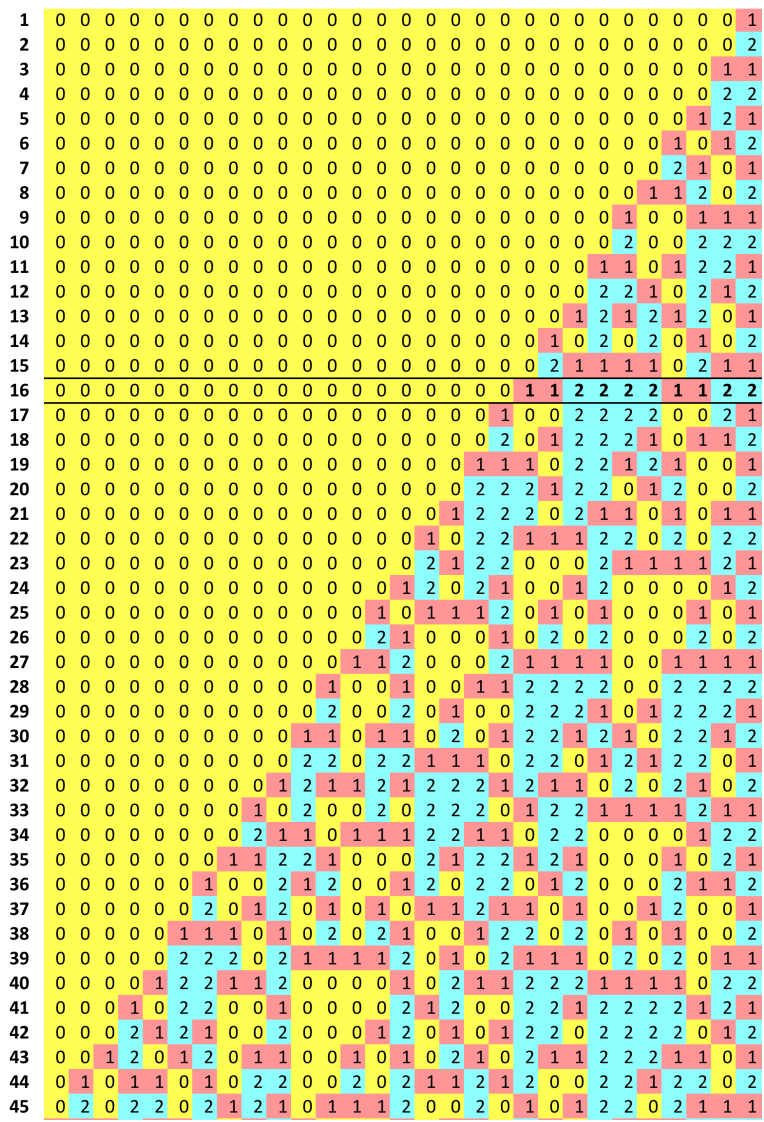}} }
\caption{Powers of $2$ in base $3$: the multiplication automaton of Example 1. It is conjectured that $2^{15}$ is the last 
power of $2$ whose base-$3$ expansion is zero-free.} 
\label{fig:PowersOf2Base3}
\end{figure}

\subsubsection*{Example 3} Note that in Definition \ref{quasicellular} we have assumed that the allowable set of primes $F$ from which 
the multipliers are chosen is {\it finite\/}. If this condition is removed, Conjecture \ref{conj:convvtoequi} becomes false, as the following example shows. 
We work in base $q=10$ here, but similar examples can be given in other bases.
Recall from elementary number theory that a {\it repunit\/} is a positive integer whose (base 10) expansion consists of a string
of ones, e.g. $1,\, 11,\, 111,\, 1111,\,\ldots$, etc. If a repunit $R$ has $n$ digits, then of course
\[
 R\;=\;\frac{10^n-1}{9}\;=\;11\cdots11 \ \ \ (n\ \textrm{times})\ .
\]
Let $k$ be any integer greater than $1$ and consider the infinite sequence of repunits $R_1<R_2<\cdots<R_n<\cdots$ given by
\[
 R_n\;=\;\frac{10^{2^n}-1}{9}
\]
Then the identity $10^{2^{n+1}}-1\;=\;\left(10^{2^n}-1\right) \left(10^{2^n}+1\right)$
shows that $R_n$ divides $R_{n+1}$, for each $n\geq 1$. Moreover, since $\mathrm{gcd}(10^{2^n}-1\,,\,10^{2^n}+1)=1$, we see that 
$R_{n+1}$ has at least one prime factor which does not divide $R_n$. Hence there exist a sequence of primes $\pi_1,\pi_2,\ldots, \pi_n,\ldots$ 
(not necessarily distinct, but ranging over infinitely many values) and a sequence of natural numbers $1=s_1<s_2<\cdots<s_n<\cdots$ 
such that $R_n=\pi_1\pi_2\cdots \pi_{s_n}$ for all $n$. 
But this means that each $R_n$ appears as a row of the (generalized) $\MA$ with base $q=10$, seed $a=1$ and with $F$ being the set of all primes appearing 
as factors of some $R_n$, which is infinite as we have shown. In particular, the line-by-line digits for this automaton cannot be
asymptotically equidistributed. Therefore Conjecture \ref{conj:convvtoequi} can fail to hold when $F$ is infinite.{\footnote{It is an old 
unsolved problem to know whether the sequence of repunits $1,\,11,\,111,\, 1111,\, 11111,\,\ldots$ written in the decimal system, 
say, contains infinitely many primes.}}

\subsection{Heuristics for convergence to equilibrium}\label{sub:Heuristic1}
Let us now give an chain of heuristic arguments lending further support in favor of Conjecture \ref{conj:convvtoequi}. 
The heuristics will be formulated for the most part in the context of Example 1 above ($q=3$ and $F=\{2\}$). 
The key idea is to model the deterministic
behaviour of an $\MA$ by a stochastic process. The irregularities observed in the 
rows of an $\MA$ are due to the carryover effect, and these will be interpreted as random. 
This will be our standing assumption. Stated informally: 
\emph{everything  is as if all would be random with appropriate distributions}.

\subsubsection{Heuristics I: Inhomogeneous Markov Chains}\label{IMCsection}

As pointed out already, we restrict our analysis almost entirely to the $\MA$ of Example 1. We make the further restriction of looking only at the evolution 
of the populations of the $3$ species given by the single digits $0,1,2$. We ignore, of course, all zeros in positions to the left of 
the last non-zero digit in each row. Although our multiplication automaton is a {\it deterministic\/} object, 
we will model its evolution -- more precisely the row-by-row evolution (or growth) of the population frequencies (or proportions) of the digits $0,1,2$ -- 
by a {\it stochastic process\/}. Once more, we leave to the reader to see how what
we propose here can be extended to the general situation of the conjecture. 

In the example at hand, the passage from row $n$ to row $n+1$ of the multiplication automaton involves only multiplication by $2$. 

The only possible values of the carryovers are $0$ and $1$; in the general case of multiplication of $\pi _i$ in base $q$, finding the maximal carryover is indeed the first computation to be done.   

Thus, for each $i\geq 0$ and each $n\geq 0$ we have either  $x_{i,n+1}=2x_{i,n} \mod 2$ if there is no carryover at the column position $i$, or $x_{i,n+1}=2x_{i,n} +1 \mod 2$ if there is a carryover at that position. This means that in each position $i$, the possible transitions of digits from one row to the next are 
$0\to 0$, $1\to 2$ and $2\to 1$ without carryover, and $0\to 1$, $1\to 0$ and $2\to 2$ with carryover. These spell out the 
{\it incidence matrix\/} 
\begin{equation}\label{MatriX3_2-0}
\bm{A} \;=\;
\begin{bmatrix} 
                   1&  1 & 0   \\
                   1&  0 & 1   \\
                   0&  1 & 1
\end{bmatrix}\ .
\end{equation}
This suggests that the evolution of the proportions $p_0^{(n)}, p_1^{(n)},p_2^{(n)}$ (with $p_0^{(n)}+p_1^{(n)}+p_2^{(n)}=1$) 
of the digits $0,1,2$ from row to row of the automaton can be modeled 
by a simple (homogeneous) Markov chain. We refer to the column vectors $\bm{p}_n= (p_0^{(n)},p_1^{(n)},p_2^{(n)})^t$ as {\it population vectors\/}. 
Our standing assumption in this context is that \emph{the allowable transitions from each given 
state have equal probabilities}. Hence this Markov chain has as its {\it transition matrix\/} the doubly stochastic matrix $\bm{P}=\frac{1}{2}\bm{A}$. 

The situation is not quite so simple, however, because this proposed scheme fails to take into account the fact that the {\it total population\/} $k_n$ 
occasionally  
increases as we go from row to row ($k_n\to\infty$ as $n\to\infty$)
depending on the leading digit, the multiplication factor, 
and the carryover inherited there.
This growth of population size is our only hope to see the asymptotic equidistribution of digits stated above. 
Indeed, we know from \S \ref{sec:tails} that the tails of size $\ell\geq 1$, namely the blocks $x_{\ell-1,n}\cdots x_{1,n}x_{0,n}$ (with $n=0,1,\ldots$) form
a {\it periodic sequence\/}. The same is true of the sequence of $\ell$-blocks $x_{\ell-1+r,n}\cdots x_{r+1,n}x_{r,n}$ for every fixed $r$. 
Hence there is no hope to observe equidistribution of digits along blocks of fixed sizes.

This compels us to change the model slightly, to accommodate the eventual increase in population size. Only two things can happen when we go from row 
$n$ to row $n+1$: either $k_{n+1}=k_n$, {\it i.e.\/}, the total population size stays the same, or $k_{n+1}=k_n+1$, {\it i.e.\/}, the total population size 
grows by $1$. In the first case, the new proportions at time $n+1$ can be obtained from the old proportions at time $n$ by multiplication with the matrix
$\bm{P}$. In the second case, we can think that the left-most digit at row $n$, namely $x_{k_n-1,n}$ gives rise not to one, but two new digits at row $n+1$, 
namely $x_{k_n-1,n+1}$ and $x_{k_n,n+1}$. Of these two, the last one is always equal to $1$. Since we either have 
$x_{k_n-1,n}=1$ or $x_{k_n-1,n}=2$, we now either have a non-zero probability of a transition $1\to 1$, or the probability of a transition 
$2\to 1$ is now slighty greater than $\frac{1}{2}$. In other words, the new population vector $\bm{p}_{n+1}$ is obtained from the old population vector
$\bm{p}_n$ by multiplication with one of the two stochastic matrices 
\begin{equation}\label{twomatrices}
\bm{Q}_n \;=\;
\begin{bmatrix} 
                   \frac{1}{2}&  \frac{1}{2} & 0   \\
                   \frac{1}{2+\epsilon_n}&  \frac{\epsilon_n}{2+\epsilon_n} & \frac{1}{2+\epsilon_n}   \\
                   0&  \frac{1}{2} & \frac{1}{2}
\end{bmatrix}\ \ \ ; \ \ 
\bm{R}_n \;=\;
\begin{bmatrix} 
                   \frac{1}{2}&  \frac{1}{2} & 0   \\
                   \frac{1}{2}&  0 & \frac{1}{2}   \\
                   0&  \frac{1+\epsilon_n}{2+\epsilon_n} & \frac{1}{2+\epsilon_n}
\end{bmatrix}\ .
\end{equation}
where $\epsilon_n=\frac{1}{k_n+1}$ accounts for the increase of the population of $1$'s by one. 

Summarizing, the evolution of population proportions in each row of the multiplication automaton in Example 1 can be modelled by 
an {\it Inhomogeneous Markov Chain\/} (IMC), given by a sequence of stochastic transition matrices $(\bm{P}_n)$ and (column) probability vectors 
$\bm{p}_{n}$ such that $\bm{p}_{n+1}^t=\bm{p}_{n}^t\bm{P}_n$. 
Here, each transition matrix $\bm{P}_n\in \{\bm{P},\bm{Q}_n,\bm{R}_n\}$. There is a well-developed theory 
of IMC's; see for instance \cite{B}, or \cite{IM}. 

As it turns out, convergence to equilibrium as expressed in Conjecture \ref{conj:convvtoequi} is tantamount, in the present example, to 
{\it strong ergodicity\/} of the IMC. Let us explain this point. For each pair of integers $n>m\geq 0$, let 
$\bm{P}(m,n)=\bm{P}_m\bm{P}_{m+1}\cdots \bm{P}_{n-1}$. Following, \cite[\S 6.8.1]{B}, we say that the IMC is {\it strongly ergodic\/} if 
there exists a row probability vector $\bm{q}$ such that
\[
 \lim_{n\to\infty} \sup_{\bm{p}}\, d_V(\bm{p}^t\bm{P}(m,n)\,,\,\bm{q}) = 0\ ,
\]
where the supremum is over all column probability vectors $\bm{p}=(p_0,p_1,p_2)^t$, and where $d_V$ is the so-called {\it distance in variation\/}, defined by
\[
 d_V(\bm{\alpha},\bm{\beta})\;=\; \frac{1}{2} \sum_{i=0}^{2} |\alpha_i-\beta_i| \ ,
\]
whenever $\bm{\alpha}=(\alpha_0,\alpha_1,\alpha_2)$ and $\bm{\beta}=(\beta_0,\beta_1,\beta_2)$ are (row) probability vectors. 
Recall (see \cite[\S 2.5]{B}, or 
any standard text on Markov chains) that a {\it stationary distribution\/} for a stochastic matrix $\bm{Q}$ is a column probability vector 
$\bm{v}$ such that $\bm{v}^t\bm{Q}=\bm{v}^t$. We have the following sufficient condition for strong ergodicity, as stated in \cite[Th.~6.8.5]{B} 
({\it cf.\/} \cite[Th. V.4.5, p. 170]{IM}). 

\begin{customthm}{C}
 If each transition matrix $\bm{P}_n$ has at least one stationary distribution and if there exists a stochastic matrix $\bm{P}_{\infty}$ such
that $\|\bm{P}_n - \bm{P}_{\infty}\|\to 0$ as $n\to\infty$, then the IMC is strongly ergodic.{\footnote{Here, $\|\cdot\|$ is the {\it max-norm\/}: 
if $\bm{A}=(a_{ij})$, then $\|\bm{A}\|=\max_{i,j} |a_{ij}|$.}}  
\end{customthm}

The hypotheses in this theorem are met in our example. Indeed, here we clearly see from \eqref{twomatrices} that $\|\bm{Q}_n-\bm{P}\|\to 0$ and
$\|\bm{R}_n-\bm{P}\|\to 0$ as $n\to \infty$; hence we can take $\bm{P}_\infty = \bm{P}$. Moreover, an easy computation shows that the probability vectors
\begin{align}\label{twovectors}
\bm{q}_n\;&=\; \left(\frac{2}{6+\epsilon_n}\,,\,\frac{2+\epsilon_n}{6+\epsilon_n}\,,\,\frac{2}{6+\epsilon_n}\right)^t \\
\bm{r}_n\;&=\; \left(\frac{2+2\epsilon_n}{6+5\epsilon_n}\,,\,\frac{2+2\epsilon_n}{6+5\epsilon_n}\,,\,\frac{2+\epsilon_n}{6+5\epsilon_n}\right)^t \nonumber
\end{align}
are stationary distributions for $\bm{Q}_n$ and $\bm{R_n}$, respectively.

The heuristic argument just presented supports the validity of Conjecture \ref{conj:convvtoequi} in the case of Example 1, as long as one accepts 
the standing assumption stated in the beginning of this section.
The formalism of IMC's can be similarly used to treat $\MA$'s with any other bases $q$ and allowable sets of multipliers $F$. 
The analysis of the two $\MA$'s in Example 2 is completely analogous, because in those two cases $F=\{3\}$ is still a unitary set.
Things become a bit different when $F$ has two or more elements. For example, take $q=5$ and $F=\{2,3\}$. 
Then instead of a single (unperturbed) stochastic matrix $\bm{P}$ we now have a pair of such matrices, one for each multiplier:
\begin{equation}\label{MatriX5-2_0}
\bm{P}_{2,5}= \frac{1}{2}
\begin{bmatrix} 
                   1&  1 & 0& 0& 0   \\
                   0&  0 & 1& 1& 0   \\
                   1& 0& 0& 0 & 1   \\
                   0& 1& 1& 0 & 0   \\
                   0& 0& 0& 1 & 1   \\
\end{bmatrix}\ \ \ ,\ \ 
\
\bm{P}_{3,5}= \frac{1}{3}
\begin{bmatrix} 
                   1&  1 & 1& 0& 0   \\
                   1&  0 & 0& 1& 1   \\
                   0& 1& 1& 1 & 0   \\
                   1& 1& 0& 0 & 1   \\
                   0& 0& 1& 1 & 1   \\
\end{bmatrix}\,.
\end{equation}
In each of these two matrices, the number of non-zero entries in each row equals the number of possible carryovers (upon multiplication 
by the corresponding prime).  

\subsubsection{Heuristics II: Block protection}\label{sub:Heuristic2}
As we have seen in the previous subsection, in the context of Example 1, each digit $x\in \mathcal{A}_3=\{0,1,2\}$ in a given position on row $n$ of our $\MA$ 
gives rise to a new digit $y\in\{y_0,y_1\}$ immediately below it on row $n+1$, where either $y=y_0=2x \mod 3$ or $y=y_1=2x+1\mod 3$, depending on whether 
the carryover at that position is $0$ or $1$, respectively. We have stated that, if the carryovers in row $n-1$ are assumed to be randomly placed, then 
the transition probabilities for the digit transitions $x\mapsto y$ are equal to $\frac{1}{2}$ (hence the stochastic matrix $\bm{P}$ introduced earlier). 
Strictly speaking, this is not correct. 
But as we will see below, the statement is close to being true provided the digits on row $n$ are already approximately uniformly distributed in the 
following (finite) sense. We say that a $k$-block $B=x_1x_2\cdots x_k\in\A_3^k$ {\it occurs\/} in an $N$-block $\omega=z_1z_2\cdots z_N\in \A_3^N$ if there exists 
$0\leq j\leq N-k$ such that $z_{j+1}=x_1\,,\,z_{j+2}=x_2\,,\,\ldots \,,\,z_{j+k}=x_k$.

\begin{defi}
 An $N$-block $\omega\in \mathcal{A}_3^N$ is said to be $k$-balanced (for a given $k$ with $1\leq k<N-3^k$) if for each 
$k$-block $B\in \mathcal{A}_3^k$ the total number of ocurrences of $B$ in $\omega$ divided by the total number of $k$-blocks occurring in $\omega$ is equal to 
$1/3^{k}$.{\footnote{Note that the total number of $k$-blocks is $3^k$.}}
\end{defi}

The rough idea is that if row $n$ of our $\MA$ is $k$-balanced and we count, for a given allowable transition $x\mapsto y$, how many times this transition occurs 
when we go from row $n$ to row $n+1$ and divide that number by the total number of observed transitions which start with the digit $x$, then this ratio -- 
which we call an {\it empirical transition probability\/} -- is approximately equal to $\frac{1}{2}$, with an error which is exponentially small in $k$. 
We call this phenomenon {\it block protection\/}. The result can be formulated as follows.

\begin{prop}\label{protectblock}
 If the $n$-th row $\{x_{i,n}\}_{1\leq i\leq k_n}$ is $k$-balanced for some $k\geq 1$, then the empirical probability $p_{xy}(n)$ of each transition 
$x\mapsto y$ (computed from the digit 
transitions $x_{i,n}\mapsto x_{i,n+1}$)  satisfies $\left|p_{xy}(n) -\frac{1}{2}\right| \leq \frac{1}{2\cdot 3^k}$.
\end{prop}
 
The proof is deferred to Appendix \ref{appendix}.

Instead of looking at transitions $x\mapsto y$ between digits, we may consider more generally transitions $X\mapsto Y$  between blocks 
$X, Y$ of a fixed length $\ell$. Here $X=x_{i+\ell-1,n}\cdots x_{i+1,n}x_{i,n}$ lies on row $n$ of our $\MA$ while  
$Y=x_{i+\ell-1,n+1}\cdots x_{i+1,n+1}x_{i,n+1}$ lies on row $n+1$ immediately below $X$. For each given $X$ there are two possible values for $Y$ 
(depending on the carryover $c_{i-1,n}\in\{0,1\}$). One can talk about block protection of such {\it blocks\/} in the same way we talked about block protection
of {\it digits\/}. An exact analogue of Proposition \ref{protectblock} holds true if we simply replace $x$ by $X$ and $y$ by $Y$, and the proof is similar. 
The end result is that {\it if row $n$ of our $\MA$ is $k$-balanced\/}, then the empirical transition probabilities $p_{XY}(n)$ differ from $\frac{1}{2}$ 
by an error smaller than $3^{-k}$ (where $k$ is the length of the protecting blocks). 
To be more specific, let $\bm{P}_{\ell, n}^{3,2}$ be the stochastic matrix that gives the actual transitions of $\ell$-blocks from row $n$ to row $n+1$, and  
{\it assume\/} that row $n$ is $k$-balanced. For $\ell=2$ there are $3^2=9$ blocks. If all allowed transitions were equally likely, we would get the 
$9\times 9$ stochastic matrix
\begin{equation}\label{MatriX3-2-2_0}
\bm{P}_{2}^{3,2}=\frac{1}{2}
\begin{bmatrix} 
                   1&  1 & 0& 0& 0 & 0& 0& 0& 0   \\ 
                   0&  0 & 1& 1& 0 & 0& 0& 0& 0   \\ 
                   0&  0 & 0& 0& 1 & 1& 0& 0& 0   \\ 
                   0&  0 & 0& 0& 0 & 0& 1& 1& 0   \\ 
                   1&  0 & 0& 0& 0 & 0& 0& 0& 1   \\ 
                   0&  1 & 1& 0& 0 & 0& 0& 0& 0   \\ 
                   0&  0 & 0& 1& 1 & 0& 0& 0& 0   \\ 
                   0&  0 & 0& 0& 0 & 1& 1& 0& 0   \\ 
                   0&  0 & 0& 0& 0 & 0& 0& 1& 1   \\ 
\end{bmatrix}\,.
\end{equation}
The actual stochastic matrices $\bm{P}_{2,n}^{3,2}$ differ from the above matrix only in its non-zero entries, and by less than $3^{-k}$, {\it provided row $n$ is 
$k$-balanced\/}.
More generally, for arbitrary $\ell$, if all allowed transitions $X\mapsto Y$ were equally likely, the associated stochastic 
matrix would be
\begin{equation}\label{MatriX3-2-b_0}
\bm{P}_{\ell}^{3,2}=\frac{1}{2}\left[\delta _{i,(2i-1)\bmod 3^\ell }+\delta _{i,2i\bmod 3^\ell}\right]_{1\leq i,j\leq 3^\ell} \ ,
\end{equation} 
where $\delta_{i,j}$ stands for Kronecker's delta. Once again, the actual stochastic matrices $\bm{P}_{\ell,n}^{3,2}$ differ 
from the matrix in \eqref{MatriX3-2-b_0} only in its non-zero entries, and by less than $3^{-k}$ {\it if row $n$ is 
$k$-balanced\/}. 

\subsubsection{Heuristics III: How to get a good start}\label{sub:Heuristic3}

Now we come to the last part of the heuristics. Almost all of what was done in parts I and II of the heuristics can 
be rigorously proved. This is not the case, so far, with the arguments in the present section. The discussion will be more informal. 

The IMC formalism as described in \S \ref{IMCsection} can only work if the stochastic matrices $\bm{P}_{\ell,n}^{3,2}$ 
introduced in \S \ref{sub:Heuristic2} 
become asymptotically closer and closer to the corresponding matrices $\bm{P}_{\ell}^{3,2}$ as $n\to \infty$, for any given $\ell$. 
Thus, we need
\begin{equation}\label{stochpell}
 \lim_{n\to\infty} \left\|\bm{P}_{\ell,n}^{3,2} - \bm{P}_{\ell}^{3,2}\right\| \;=\;0 \ ,
\end{equation}
for any block-length $\ell$. This is indeed corroborated by our computational evidence.  Proposition \ref{protectblock} gives us a 
much weaker, conditional result: {\it if\/} the $n$-th row $N_n$ of our 
$\MA$ is $k$-balanced, {\it then\/} $\|\bm{P}_{\ell,n}^{3,2} - \bm{P}_{\ell}^{3,2}\|\leq 3^{-k}$. It is of course too much to expect 
that a row of our $\MA$ will be perfectly $k$-balanced. But it is not difficult to generalize Proposition \ref{protectblock} so that 
the expression ``$k$-balanced'' is replaced  by a suitable concept of ``almost $k$-balanced'' (which can be stated in an 
appropriate quantitative way), so that the estimate in the conclusion is almost the same, with $3^{-k}$ replaced by $O(3^{-k})$, say. 
Let us agree to say that the matrix $\bm{P}_{\ell,n}^{3,2}$ is {\it $k$-reasonable\/} if $\|\bm{P}_{\ell,n}^{3,2} - \bm{P}_{\ell}^{3,2}\|= O(3^{-k})$. 
The stability of the arguments presented so far then tell us that if row $n$ of our $\MA$ is almost $k$-balanced, then the stochastic matrices 
$\bm{P}_{\ell,n+j}^{3,2}$, $j=0,1,\ldots,s-1$  remain $k$-reasonable for a certain number $s$ of steps.

Thus, in order to justify the use of the IMC formalism and get a proof of Conjecture \ref{conj:convvtoequi} (at least for the case of the $\MA$ of Example 1), 
we need two things:
\begin{enumerate}
\item[(1)] To guarantee the possibility of a {\it good start\/}, {\it i.e.,}  a row $n$ of our automaton which is almost $k$-balanced, 
for as large a value of $k$ as possible.  
\item[(2)] To make sure that
throughout the steps $j=0,1,\ldots, s-1$ 
the longer blocks in $N_{n+j}$ (\emph{i.e.,} those whose evolutions are not protected by the distributions of even longer blocks) 
are not progressively getting too unevenly distributed, with consequences that would then cascade down to shorter and 
shorter blocks.  
\end{enumerate}

At least at a rough level, point (1) is not difficult to achieve thanks to our results 
in \S \ref{sec:erg1}. Indeed, by what we have seen in that section, adding a decimal point in front of $N_n$ (the $n$-row of our $\MA$), we get the orbit
$x_n=T_{2,3}^n(x_0)$ of $x_0=[0.1]_3$ under the circle map $T_{2,3}$. Since all orbits of $T_{2,3}$ are dense, given any $k$-balanced string 
$Y=y_1y_2\cdots y_m\in \mathcal{A}_3^m$ (with $y_1\neq 0$), there exists $n\geq 1$ such that $x_n$ has the string $Y$ as its 
$m$-prefix, or head, which we call {\it good\/}. For each 
such $n$ with a good head, the corresponding $N_n$ has a suffix, or tail, of a certain length $r$. Note that we cannot choose $r$ {\it a-priori}: if one 
uses the density of orbits of $T_{2,3}$,  one is forced to accept the value of $r$ imposed by the choice of the good head $Y$.

Now, as we iterate further and consider the 
successive rows $N_n, N_{n+1}, \ldots,\break N_{n+j},\ldots$ of the automaton, the tail of length $r$ generates a periodic sequence of tails 
lying beneath it (all with the same length $r$), whereas 
the heads increase in size. Clearly, the constant character of $r$ as $j$ increases makes the 
tails irrelevant in the computation of asymptotic proportions of symbols or other blocks. The problem then becomes point (2) 
above for the heads. This point is experimentally verified but is mathematically beyond our reach at this writing, 
contrary to point (1) as formulated above.




\subsection{Appendix: Proof of Proposition \ref{protectblock}}\label{appendix}

As promised, in this section we prove Proposition \ref{protectblock}. 
The proof will require the two lemmas presented below. Let us fix $k\geq 1$. It is an easy consequence of Proposition \ref{prop:locality} that, for each 
$i>k$ and each $n\geq 1$, the element $x_{i,n+1}$ of our $\MA$ is completely determined by the following data:
\begin{enumerate}
 \item The element $x_{i,n}$ lying immediately above $x_{i,n+1}$;
 \item The $k$-block $B_{i,n,k}=x_{i-1,n}x_{i-2,n}\cdots x_{i-k,n}\in \A_3^k$ lying to the right of $x_{i,n}$;
 \item The carryover $c_{i-k-1,n}$, {\it i.e.\/} the carryover immediately to the right of $B_{i,n,k}$.t
\end{enumerate}
In other words, there exists a function $\Phi_k:\,\A_3\times \A_3^k\times \{0,1\} \to \A_3$ such that
\begin{equation}\label{phik}
 x_{i,n+1}\;=\; \Phi_k\left(x_{i,n}\,,\,B_{i,n,k}\,,\,c_{i-k-1,n}\right)\ .
\end{equation}
In fact, the function $\Phi_k$ can be explicitly computed. Define the {\it value\/} of a $k$-block $B=x_1x_2\cdots x_k$ by 
\begin{equation}\label{valueblock}
 v(B)\;=\;\sum_{j=1}^k x_j3^{k-j}\ .
\end{equation}
Note that every block $B$ is uniquely determined by its value. We now have the following lemma.
\begin{lemma}
 For each digit $x\in \A_3$, each $k$-block $B=x_1x_2\cdots x_k\in \A_3^k$ and each carryover $c\in\{0,1\}$, we have
\begin{equation}\label{phiformula}
 \Phi_k(x,B,c)\;=\; 2x + \left\lfloor \frac{2v(B)+c}{3^k} \right\rfloor \ \ \mod 3 \ .
\end{equation}

\end{lemma}

\begin{proof}
 The concatenated block $xB=xx_1x_2\cdots x_k$ has value $v(xB)=3^kx+v(B)$. In order to compute $y=\Phi_k(x,B,c)$, we multiply this value 
by $2$ and add the carryover $c$, getting the number $w=2v(xB)+c$. This number is written in base $3$ and the resulting block is placed 
beneath $xB$. The digit $y$ (immediately below $x$) is precisely the $k$-th digit of $w$ from right to left, {\it i.e.,} $y=\lfloor w/3^k\rfloor$.
Hence we have
\begin{align*}
 y\;&=\; \left\lfloor \frac{2(3^kx+v(B))+c}{3^k}\right\rfloor \ \ \mod 3 \\
    &=\; 2x + \left\lfloor \frac{2v(B)+c}{3^k}\right\rfloor \ \ \mod 3\ ,
\end{align*}
and this proves \eqref{phiformula}.
\end{proof}
Note that the last term in the right-hand side of \eqref{phiformula} is equal to either $0$ or $1$. 

Now, it turns out that $\Phi_k$ is ``almost'' independent of the variable $c\in\{0,1\}$. 
Roughly speaking, the only way the carryover $c$ to the right of the $k$-block $B$ can influence the value of the digit $y$ immediately 
below $x$ (on the left of $B$) is if $B$ happens to be the block $\mathbf{1}_k=11\cdots 1$ ($k$ times). Every other block will contain 
in some position a $0$ or a $2$; upon multiplication by $2$ these yield $0$ and $1$, respectively, and any carryover effect coming from the right of that 
position will not go through to the left of it. This is part (i) of the following lemma. 

\begin{lemma}\label{phiproperties}
 The function $\Phi_k$ has the following properties.
\begin{enumerate}
 \item[(i)] If $x\in \A_3$ and $B\in \A_3^k$, then $\Phi_k(x,B,0)\neq \Phi_k(x,B,1)$ if and only if $B=\mathbf{1}_k$;
 \item[(ii)] If $x\in \A_3$ and $y\in\{y_0,y_1\}$, then 
\[
 \#\left\{ B\in \A_3^k\setminus\{\mathbf{1}_k\}\,:\;\Phi_k(x,B,c)=y\right\}\;=\; \frac{3^k-1}{2}
\]
\end{enumerate}
\end{lemma}

\begin{proof}
 First, let $B$ be a $k$-block such that $\Phi_k(x,B,0)\neq \Phi_k(x,B,1)$. Then from \eqref{phiformula} we see that
\begin{equation}\label{phiprop1}
 0\;=\;\left\lfloor \frac{2v(B)}{3^k} \right\rfloor \;\neq\; \left\lfloor \frac{2v(B)+1}{3^k} \right\rfloor \;=\;1\ .
\end{equation}
Hence we must have simultaneously
\[
  \frac{2v(B)}{3^k} \;<\;1 \ \ \mathrm{and}\ \ \ \frac{2v(B)+1}{3^k} \;\geq\; 1
\]
From these two inequalities we deduce that $3^k-1\leq 2v(B)< 3^k$, and therefore $v(B)=\frac{1}{2}(3^k-1)$. This means that 
$B=\mathbf{1}_k$. Conversely, if $B=\mathbf{1}_k$ then by a simple computation we see that \eqref{phiprop1} holds true. This proves (i). 

To prove (ii), there are two cases to consider: $y=y_0$ and $y=y_1$. In either case, if $B\in \A_3^k\setminus\{\mathbf{1}_k\}$ then by part (i) we have
\begin{equation}\label{phipro2}
 \left\lfloor \frac{2v(B)}{3^k} \right\rfloor \;=\; \left\lfloor \frac{2v(B)+1}{3^k} \right\rfloor
\end{equation}
If $y=y_0$, then both sides in \eqref{phipro2} are equal to $0$, and we get  $0\leq v(B)< \frac{1}{2}(3^k-1)$. The number of blocks $B$ satisfying these 
inequalities is precisely $\frac{1}{2}(3^k-1)$. If $y=y_1$, then both sides in \eqref{phipro2} are equal to $1$, and this time we deduce that 
$\frac{1}{2}(3^k+1)\leq v(B)\leq 3^k-1$. The number of blocks $B$ satisfying these last 
inequalities is also $\frac{1}{2}(3^k-1)$. This finishes the proof.
\end{proof}

\phantom{Here is some space at the top}

\phantom{Here is some more space}
\begin{figure}[t]
\hbox to \hsize {
{\hspace{-11.0em} \includegraphics[width=10.0cm]
{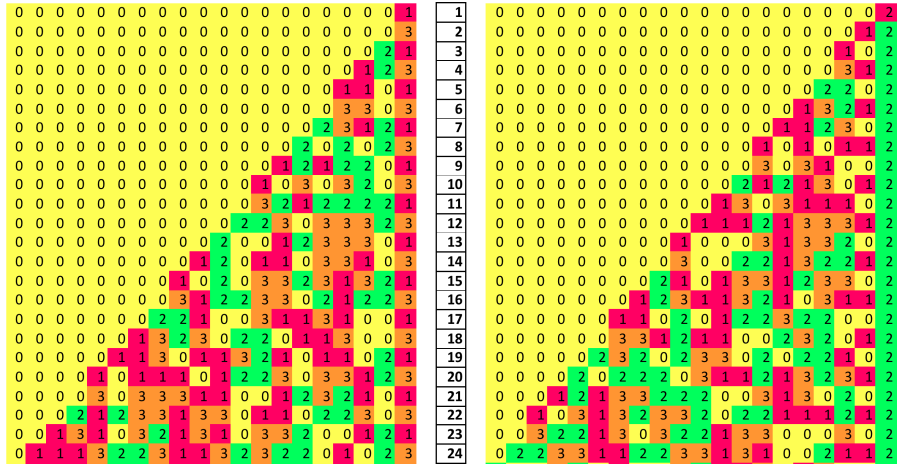}} }
\caption{The two multiplication automata of Example 2. } 
\label{fig:HeadsBase4}
\end{figure}
\bigskip

\subsubsection*{Proof of Proposition \ref{protectblock}}
Consider an allowable transition of digits $x\mapsto y$ from row $n$  to row $n+1$ of our $\MA$. We are assuming that row $n$ is $k$-balanced. 
In order to compute the empirical probability $p_{xy}(n)$ of such transition, we need to count how many times this transition happens and 
divide it by the total number of transitions which start with $x$ on row $n$.  For this purpose, first we count how many blocks 
$B\in \A_3^k\setminus\{\mathbf{1}_k\}$ are such that $\Phi_k(x,B,c)=y$. The answer is given by Lemma \ref{phiproperties}(ii): 
there are $\frac{1}{2}(3^k-1)$ such blocks. Since row $n$ of our $\MA$ is $k$-balanced, the proportion of such blocks in that row is 
therefore the quotient $\frac{1}{2}(3^k-1)/3^k$. This already tells us that \begin{equation}\label{empirc1} p_{xy}(n) \geq \frac{1}{2}-\frac{1}{2\cdot 3^k}
\end{equation}
We still have to account for the occurences of the block $B=\mathbf{1}_k$. Let $c\in \{0,1\}$ be such that $\Phi_k(x,\mathbf{1}_k,c)=y$. 
Each occurrence of $B=\mathbf{1}_k$ on row $n$ for which the carryover immediately to the right of $B$ equals $c$ contributes to the desired 
empirical probability. Since the proportion of such occurrences is at most $\frac{1}{3^k}$, this shows that 
\begin{equation}\label{empirc2}
 p_{xy}(n)\leq \frac{1}{2}-\frac{1}{2\cdot 3^k}+ \frac{1}{3^k}= \frac{1}{2}+\frac{1}{2\cdot 3^k}\ . 
\end{equation}
Combining \eqref{empirc1} with \eqref{empirc2} we get the inequality in the statement. \hfill\qed


\phantom{Here is some space at the top}


\phantom{Here is some space at the top}
\bigskip
\bigskip

\begin{figure}[b]
\hbox to \hsize {
{\hspace{-11.0em} \includegraphics[width=8cm]
{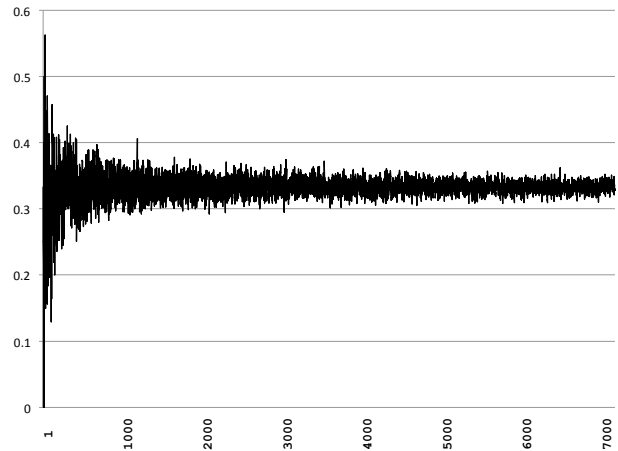}} }
\caption{This figure plots the proportion of the digit $0$ in the base-$3$ expansion of $2^n$ as a function of $n$; note the 
apparent convergence to $\frac{1}{3}$. } 
\label{fig:density10}
\end{figure}


\newpage

\phantom{Here is some space at the top}
\bigskip

\begin{figure}[hb]
\hbox to \hsize {
{\hspace{-11.0em} \includegraphics[width=7cm]
{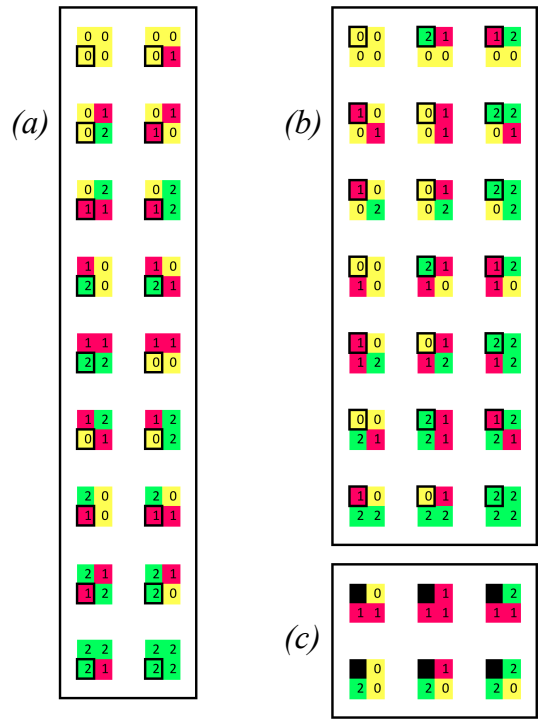}} }
\caption{Carryover structure for the two automata (horizontal and vertical) associated to Example 1.} 
\label{fig:qauto}
\end{figure}

\newpage



%
\section*{Acknowledgements}
The authors are grateful to George Hentchel for bringing this problem to the attention of C.T., 
and for sharing insights on how asymptotic equidistribution of digits would solve Sloane's problem in base $3$.


\begin{thebibliography}{999}
\bibliographystyle{plain}

\bibitem[B]{B} 
P.~Br\'emaud. \emph{Markov Chains: Gibbs Fields, Monte Carlo Simulation and Queues\/}. Texts in Applied Mathematics {\bf{31}},
Springer-Verlag, 1999.

\bibitem[Bo]{Bo} 
M.~Boshernitzan. \emph{Dense orbits of rationals\/}. Proc.~Amer.~Math.~Soc. {\bf{117}} (1993), 1201--1203.

\bibitem[G]{G} 
R. K.~Guy, \emph{Unsolved Problems in Number Theory\/}. 2nd ed., Springer-Verlag, New York, 1994.

\bibitem[He]{He} G.~Hedlund. \emph{Endomorphisms and automorphisms of the shift dynamical systems\/}. Mathematical System Theory {\bf{3}}, (1969) 320--375.

\bibitem[H]{H} H.J.~Hinden. \emph{The additive persistence of a number\/}.  Journal of Recreational Mathematics 
{\bf{7}} (1974), 134--135.

\bibitem[IM]{IM} D.L.~Isaacson \&\ R.W.~Madsen. \emph{Markov Chains: Theory and Applications}. John Wiley and Sons, New York, 1976.

\bibitem[K]{K} 
G.~Keller. \emph{Equilibrium States in Ergodic Theory\/}. London Mathematical Society Student Texts {\bf{42}}, Cambridge University Press, 1998.

\bibitem[L]{L} 
J.C.~Lagarias. \emph{Ternary expansions of powers of 2\/}. J. London Math. Soc. (2) {\bf{79}} (2009), 562--588.

\bibitem[LeV]{LeV} 
W.~LeVeque. \emph{Fundamentals of Number Theory\/}.  Addison-Wesley, Reading, Massachusetts, 1977.

\bibitem[Li]{Li} 
I.~Liousse. \emph{PL Homeomorphisms of the circle that are piecewise $C^1$ conjugate to irrational rotations}.  Bull.~Braz.~Math.~Soc. {\bf{35}} (2004), 269--280. 

\bibitem[M]{M} 
P.~Mih\u{a}ilescu. \emph{Primary cyclotomic units and a proof of Catalan's conjecture}. J.~Reine~Angew.~Math. {\bf{572}} (2004), 167--195.

\bibitem[P]{P} 
M.~Pivato, \emph{Multiplicative cellular automata on nilpotent groups: structure, entropy, and asymptotics}. J. Statist. Phys. {\bf 110} (2003), 247--267.

\bibitem[S]{S} 
N.~Sloane. \emph{The persistence of a number\/}. Journal of Recreational Mathematics {\bf{6}} (1973), 97--98.

\bibitem[Tr]{Tr}
C.~Tresser, \emph{Bounding the errors for convex dynamics on one or more polytopes}. Chaos  \textbf{17} (2007), 33--49.

\bibitem[W]{W} P.~Walters. \emph{An introduction to Ergodic Theory\/}. Graduate Texts in Mathematics {\bf{79}}, 
Springer Verlag, New York, 1982.

\bibitem[T]{T} T.~Tao. \emph{The Collatz conjecture, Littlewood-Offord theory, and powers of 2 and 3\/}. Blog entry in 
{\tt{http://terrytao.wordpress.com\/}}, 2011. 

\end{thebibliography}
\end{document}